\documentclass[a4paper,11pt]{amsart}

\usepackage{amsmath,amsthm,amssymb,mathtools,thmtools,thm-restate}
\usepackage[bookmarks=false]{hyperref}

\usepackage{tikz-cd}
\usepackage{tikz}

\usepackage{enumerate}
\usepackage{enumitem}
\usepackage{color}
\usepackage{bbm}

\numberwithin{equation}{section}

\theoremstyle{plain}
\newtheorem{mthm}{Theorem}

\newtheorem{mprop}[mthm]{Proposition}
\newtheorem{theorem}{Theorem}[section]

\newtheorem{lemma}[theorem]{Lemma}
\newtheorem{proposition}[theorem]{Proposition}
\newtheorem{corollary}[theorem]{Corollary}

\theoremstyle{definition}
\newtheorem{definition}[theorem]{Definition}
\newtheorem*{definition*}{Definition}
\newtheorem{example}[theorem]{Example}
\newtheorem{remark}[theorem]{Remark}
\newtheorem{mdef}[mthm]{Definition}
\newtheorem{mex}[mthm]{Example}

\newcommand{\C}{\mathbb{C}}\newcommand{\cC}{\mathcal{C}}

\newcommand{\bF}{\mathbb{F}}
\newcommand{\cG}{\mathcal{G}}

\newcommand{\cI}{\mathcal{I}}

\newcommand{\bM}{\mathbb{M}}
\newcommand{\N}{\mathbb{N}}

\newcommand{\R}{\mathbb{R}}
\newcommand{\cS}{\mathcal{S}}

\newcommand{\cU}{\mathcal{U}}

\newcommand{\Z}{\mathbb{Z}}

\newcommand{\eps}{\varepsilon}
\newcommand{\vphi}{\varphi}

\newcommand{\norm}[1]{\left\|#1\right\|}
\newcommand{\abs}[1]{\left|#1\right|}
\newcommand{\sub}{\subseteq}

\newcommand{\ot}{\otimes}
\renewcommand{\iff}{\Leftrightarrow}

\newcommand{\Sym}{\operatorname{Sym}}
\newcommand{\id}{\operatorname{id}}

\newcommand{\G}{\mathcal{G}}
\newcommand{\U}{\operatorname{U}}
\newcommand{\dG}{\operatorname{d}_G}
\newcommand{\dGkn}{\operatorname{d}_{G_{k_n}}}
\newcommand{\dn}{\operatorname{d}_n}
\newcommand{\dkn}{\operatorname{d}_{k_n}}
\newcommand{\Dn}{\operatorname{D}_n}
\newcommand{\DGkn}{\operatorname{D}_{G_{k_n}}}
\newcommand{\dHamm}{\operatorname{d}_{\operatorname{Hamm}}}

\newcommand{\Tr}{\operatorname{Tr}}
\newcommand{\Sub}{\operatorname{Sub}}
\newcommand{\Stab}{\operatorname{Stab}}
\newcommand{\Fix}{\operatorname{Fix}}

\DeclareSymbolFont{Symbols}{OMS}{cmsy}{m}{n}
\DeclareMathSymbol{\Emptyset}{\mathord}{Symbols}{"3B}

\hypersetup{
  colorlinks   = true, 
  urlcolor     = blue, 
  linkcolor    = blue, 
  citecolor    = red   
}

\usepackage{xcolor}

\begin{document}

\title[Local HS-stability]
{Local Hilbert--Schmidt stability}

\author[F. Fournier-Facio, M. Gerasimova and P. Spaas]{Francesco Fournier-Facio, Maria Gerasimova and Pieter Spaas}
\address{Department of Pure Mathematics and Mathematical Statistics, University of Cambridge, United Kingdom}
\email{ff373@cam.ac.uk}
\address{Mathematical Institute, University of M\"unster, Einstenstrasse 62, 48149, M\"unster, Germany}
\email{mari9gerasimova@mail.ru}
\address{Department of Mathematical Sciences, University of Copenhagen, Universitetsparken 5, DK-2100 Copenhagen \O, Denmark}
\email{pisp@math.ku.dk}
\thanks{F.F.F. was supported by the Herchel Smith Postdoctoral Fellowship Fund. P.S. was partially supported by a research grant from the Danish Council for Independent Research, Natural Sciences, and partially by MSCA Fellowship No. 101111079 from the European Union. M.G.\ was supported by the DFG -- Project-ID 427320536 -- SFB 1442, and under Germany's Excellence Strategy EXC 2044 390685587, Mathematics M{\"u}nster: Dynamics--Geometry--Structure.}

\begin{abstract}
    We introduce a notion of local Hilbert--Schmidt stability, motivated by the recent definition by Bradford of local permutation stability,
    and give examples of (non-residually finite) groups that are locally Hilbert--Schmidt stable but not Hilbert--Schmidt stable.
    For amenable groups, we provide a criterion for local Hilbert--Schmidt stability in terms of group characters, by analogy with the character criterion of Hadwin and Shulman for Hilbert--Schmidt stable amenable groups.
    Furthermore, we study the (very) flexible analogues of local Hilbert--Schmidt stability, and we prove several results analogous to the classical setting. Finally, we prove that infinite sofic, respectively hyperlinear, property (T) groups are never locally permutation stable, respectively locally Hilbert--Schmidt stable. This strengthens the result of Becker and Lubotzky for classical stability, and answers a question of Lubotzky.
\end{abstract}

\maketitle

\section{Introduction}

Stability can be traced back to the famous question of whether two matrices that almost commute with respect to some given norm, must be close to two commuting matrices. This was first explicitly asked by Rosenthal \cite{rosenthal} for the normalized Hilbert--Schmidt norm and by Halmos \cite{halmos} for the operator norm. For the operator norm, the answer turned out to be negative for unitary matrices by the famous result of Voiculescu \cite{voiculescu:matrices}, but positive for self-adjoint matrices \cite{Lin, FR:Lin}. For the normalized Hilbert--Schmidt norm, the answer turns out to be positive \cite{Glebsky:commutingHS}. We refer to the introduction of \cite{ioana:comm} for an overview and further references. 

Note that the aforementioned questions for (almost) commuting unitary matrices can also be phrased in terms of (approximate) unitary representations of $\Z^2$, and in recent years there has been a surge of interest in various notions of \emph{stability} for more general groups. To make this precise, one considers a family of groups $\G$ endowed with bi-invariant metrics. The stability question for a group $\Gamma$ with respect to $\G$ then asks whether all ``approximate'' homomorphisms $\Gamma \to G \in \G$ are ``close'' to genuine homomorphisms, where distances are measured with respect to the given metrics on $G\in\G$. Two notable examples are the family $\cS$ of finite symmetric groups endowed with the normalized Hamming distance, and the family $\cU$ of finite-dimensional unitary groups endowed with the normalized Hilbert--Schmidt norm, i.e. the norm induced by the normalized trace. We refer to stability of $\Gamma$ with respect to these families as \emph{P-stability} and \emph{HS-stability}, respectively.

Despite the fact that stability is generally hard both to prove and disprove, there has been quite some progress in recent years, especially for HS-stability and P-stability. The study of HS-stability for general countable discrete groups was initiated by Hadwin and Shulman \cite{HS1,HS2}. They established a character criterion for amenable groups, and used it to establish HS-stability for, among others, virtually abelian groups. Their criterion has since been used to establish further families of examples of amenable HS-stable groups, including all virtually nilpotent groups \cite{ES:HS,levit:vigdorovich}. In the non-amenable realm, examples of HS-stable groups include virtually free groups \cite{GS:virtuallyfree}, certain graph product groups \cite{Atkinson:graphs}, one-relator groups with non-trivial center \cite{HS2}, and certain amalgamated free products and HNN extensions \cite{GS:amalgams}.

The first steps in the study of P-stability were already undertaken by Glesbky and Rivera in \cite{glebskyrivera}, who studied stability from the point of view of equations involving permutations. Arzhantseva and P\u{a}unescu observed in \cite{AP:commuting} that this problem can be rephrased in group-theoretic terms, introducing the current notion of stability, and proved that finitely generated abelian groups are P-stable. Since then, P-stability has been under intense investigation. 
For instance, we point out \cite{BLT:IRS}, where the authors prove a criterion for P-stability of amenable groups in terms of invariant random subgroups, which has since led to several new classes of examples \cite{zheng, infinitelypresented:Pstable, uncountably:Pstable}. 

Most negative examples are found in the realm of property (T) groups. Firstly, Becker and Lubotzky proved in \cite{BL:T} that infinite sofic groups with property (T) are never P-stable, and similarly that infinite hyperlinear groups with property (T) are never HS-stable. They establish this via a small perturbation in the degree of the involved symmetric groups, respectively the dimension of the matrix algebras, which prompted them to define the notion of \emph{flexible stability}. This has since become a subject of study as well, see, e.g. \cite{surface, GS:amalgams}. Whereas flexible stability of most property (T) groups is open, it was established in \cite{ISW:Coh} that $\Z^2\rtimes\operatorname{SL}_2(\Z)$ is not flexibly HS-stable. We also mention \cite{ioana:comm}, where Ioana proves that $\bF_2\times\bF_2$ is not flexibly HS-stable, which as far as the authors know is the only example that does not require some version of property (T).

An elementary observation due to Arzhantseva and P\u{a}unescu \cite{AP:commuting} is that a sofic P-stable group must be residually finite. Similarly, a hyperlinear HS-stable group must be MAP, hence by Mal'cev's Theorem \cite{malcev} it must be residually finite if it is finitely generated.
Finding a non-residually finite stable group is therefore a possible strategy towards finding a non-sofic or a non-hyperlinear group. This strategy was achieved for approximability by some other families of metric groups \cite{frobenius:approx, shatten:approx}, yet is wide open for soficity and hyperlinearity (though see e.g. \cite{bowen:burton,alon} for some connections). Nevertheless, the study of P-stability and HS-stability has largely been restricted to residually finite groups, thus leaving many relevant groups outside of its scope.

The main goal of this paper is to initiate the study of \emph{local HS-stability}, by analogy with local P-stability introduced by Bradford, cf. \cite[Problem 6.6]{bradford}. Here one only requires approximate homomorphisms to be close in normalized Hilbert--Schmidt norm to so-called \emph{partial homomorphisms}:

\begin{mdef}[{see Definitions~\ref{def:partialapprox} and \ref{def:stable}}]
    A countable discrete group $\Gamma$ is \emph{locally HS-stable} if for any sequence of maps $\vphi_n:\Gamma\to \U(k_n)$ such that 
    \[
    \forall g,h\in\Gamma: \norm{\vphi_n(gh) - \vphi_n(g)\vphi_n(h)}_2\to 0,
    \]
    there exists a sequence of maps $\psi_n:\Gamma\to\U(k_n)$ such that
    \begin{enumerate}
        \item For all $g,h\in\Gamma$, there exists $N\in\N$ such that for all $n\geq N$: $\psi_n(gh) = \psi_n(g)\psi_n(h)$.
        \item For all $g\in\Gamma$: $\norm{\vphi_n(g) - \psi_n(g)}_2\to 0$.
    \end{enumerate}
\end{mdef}

Notably, locally HS-stable hyperlinar groups (or locally P-stable sofic groups) need not be residually finite, and instead only satisfy a weaker property, namely \emph{local embeddability into finite groups (LEF)} \cite{LEF}, see Definition~\ref{def:LEF} and Lemma~\ref{lem:stable:LEF}. Intuitively, a group is LEF if there are maps to finite groups for which the homomorphism relation holds on larger and larger finite subsets, but not necessarily on the entire group. Many groups of interest are LEF, and in many ways it is a more robust notion than residual finiteness. For instance, the standard wreath product of two LEF groups is LEF \cite{LEF}, while the standard wreath product of two infinite residually finite groups is only residually finite when the base group is abelian \cite{gruenberg}. The notion of LEF is often useful when studying approximation properties of finitely generated groups that are not finitely presented, such as topological full groups of minimal Cantor systems \cite{GM:TopFullGp}.

Similar to ``classical'' HS-stability (and P-stability), the study of local HS-stability is a little more concrete in the amenable setting. Specifically, our first main result gives a criterion for local HS-stability of amenable groups in terms of characters, in line with the analogous result for HS-stability \cite[Theorem~4]{HS2}. We refer the reader to Sections~\ref{sec:prelim} and \ref{sec:stability} for the relevant definitions.

\begin{mthm}[Theorem \ref{thm:char}]
\label{intro:thm:char}

Let $\Gamma$ be a countable amenable group with a fixed epimorphism $\pi : \bF \to \Gamma$ where $\bF$ is a free group. Then $\Gamma$ is locally HS-stable if and only if the following holds:

For every character $\chi : \Gamma \to \C$ there exists a sequence $(\Lambda_n)_n$ of $\bF$-marked groups converging to $\Gamma$, together with finite-dimensional representations $\theta_n : \Lambda_n \to \U(k_n)$ such that $\tau_{k_n} \circ \theta_n$ converges pointwise to $\chi$, when all are viewed as characters of $\bF$.
\end{mthm}

In the above theorem, $\tau_{k_n}$ denotes the normalized trace on the finite-dimensional unitary group $\U(k_n)$. Moreover, we point out that there is a useful equivalent formulation of the condition in the above theorem in terms of partial homomorphisms to unitary groups, see Lemma~\ref{rem:equivalent:criterion}. In either formulation, we note that, contrary to the criterion for HS-stability, only ``local'' representations of $\Gamma$ show up, and it is perfectly possible that $\Gamma$ itself admits no nontrivial finite-dimensional representations. Indeed, it is an easy observation (Proposition~\ref{prop:directedunion}) that local stability is a local property, i.e. it is preserved under directed unions. Using this we obtain the following examples:

\begin{mex}[Corollary~\ref{cor:locallyvirtuallynilpotent}]
    Every locally finite group is locally HS-stable. In particular, the groups $S_\infty$ and $A_\infty$ of finitely supported (even) permutations of $\N$ are locally HS-stable. More generally, all locally virtually nilpotent groups are locally HS-stable.
\end{mex}

We note that both $S_\infty$ and $A_\infty$ are not MAP, and are therefore not HS-stable (Example \ref{ex:unstable:locallystable}).

Next, we establish several basic properties and results about local HS-stability for amenable groups, analogous to known results in the setting of classical HS-stability. For this we will use, among other things, our character criterion from Theorem~\ref{intro:thm:char} for (1) below, a lifting property for amenable von Neumann algebras from \cite{IS-exotic} for (2), and diagonal products in the space of marked groups together with properties of nuclear $C^*$-algebras for (3). We refer to the corresponding results below, and to Sections~\ref{sec:stability} and \ref{sec:amenable} for definitions. Notably, we formulate (very) flexible versions of local HS-stability in Definition~\ref{def:flexibly}, by analogy with the classical notions.

\begin{mthm}
\label{intro:thm:amenable}

Let $\Gamma$ be a countable amenable group.
\begin{enumerate}
    \item \emph{(Proposition \ref{prop:flexible})} $\Gamma$ is locally HS-stable if and only if it is flexibly locally HS-stable.
    \item \emph{(Proposition \ref{prop:product:amenable})} If $\Sigma$ is any countable discrete group, and $\Gamma$ and $\Sigma$ are ((very) flexibly) locally HS-stable, then so is $\Gamma \times \Sigma$.
    \item \emph{(Theorem \ref{thm:veryflex})} $\Gamma$ is very flexibly locally HS-stable if and only if it is LEF.
    \item \emph{(Proposition \ref{prop:weaklylocallyHSstable})} $\Gamma$ is weakly locally HS-stable if and only if it is LEF.
\end{enumerate}
\end{mthm}

Next, we will apply the character criterion from Theorem~\ref{intro:thm:char} to establish local HS-stability of groups of great interest in geometric and combinatorial group theory:

\begin{mthm}[Theorem~\ref{thm:full}]
\label{intro:thm:full}

Let $T$ be a minimal Cantor subshift. Then the derived subgroup $[[T]]'$ of the topological full group of $T$ is locally HS-stable. In particular, there are uncountably many non-isomorphic finitely generated groups which are locally HS-stable, but not HS-stable.
\end{mthm}

Note that the analogous result for permutations was established by Bradford \cite[Theorem 1.8]{bradford}, relying on Zheng's classification of IRS's of such groups \cite{zheng}. The way we obtain Theorem \ref{intro:thm:full} is via a recently announced result of Dudko--Medynets \cite{DM}, stating that all extremal characters of derived subgroups of topological full groups come from p.m.p. actions, together with the following general fact, which is elementary but useful and of independent interest.

\begin{mprop}[Corollary \ref{cor:PtoHS}]
\label{intro:prop:PtoHS}

Let $\Gamma$ be an amenable group such that every extremal character can be realized as the fixed-point character of a p.m.p. action. If $\Gamma$ is (locally) P-stable, then it is (locally) HS-stable.
\end{mprop}

In recent work~\cite{enrichments}, Dogon, Levit, and Vigdorovich found a further application of the character criterion of Theorem~\ref{intro:thm:char} to alternating enrichments $\mathrm{Alt}_\infty(\Gamma) \rtimes \Gamma$ (also known as lampshufflers). In particular, they prove that $\Gamma$ is locally HS-stable if and only if $\mathrm{Alt}_\infty(\Gamma) \rtimes \Gamma$ is locally HS-stable \cite[Theorem 1.10]{enrichments}. This provides a different construction of uncountably many non-isomorphic finitely generated groups which are locally HS-stable but not HS-stable.

Finally, we establish the following negative result, showing that property (T) groups are essentially never locally HS- or P-stable. This result is also new for local P-stability, and answers Question~6.5 from \cite{bradford}, which is attributed to Lubotzky.

\begin{mthm}[Theorem \ref{thm:T}]
\label{intro:thm:T}

Let $\Gamma$ be an infinite group with property (T). If $\Gamma$ is sofic, then it is not locally P-stable. If $\Gamma$ is hyperlinear, then it is not locally HS-stable.
\end{mthm}

Besides introducing local HS-stability and proving the aforementioned results, our second main goal is to provide a general framework for \emph{local stability}, which streamlines many of the arguments and allows us to establish several general results for local stability with respect to any family of metric groups. This is achieved by developing and unifying various different points of view: local embeddings, partial homomorphisms, homomorphisms to different types of ultraproducts, and limits in the space of marked groups. We introduce the general definitions and properties in Sections~\ref{sec:prelim} and \ref{sec:stability}, before focusing on local HS-stability for the remainder of the paper. In particular, we include most of the details regarding the aforementioned equivalent viewpoints to make the paper mostly self-contained.  We hope that this general framework will facilitate future work on local stability with respect to other families of metric groups.

\subsection*{Acknowledgements}

The authors wish to thank Tianyi Zheng and Henry Bradford for several helpful conversations about the contents of the paper, and Tianyi Zheng for also pointing out some useful references. The authors also wish to thank Alon Dogon, Adrian Ioana, Srivatsav Kunnawalkam Elayavalli, Arie Levit and Tatiana Shulman for useful comments on an earlier draft of the paper, and Artem Dudko and Konstantin Medynets for sharing their result from \cite{DM}. Additionally, we would like to thank Kyoto University for hosting two of the authors for the 8th KTGU Mathematics Workshop for Young Researchers, where this project was started. Finally, the authors wish to thank the anonymous referee for providing useful comments that allowed us to improve the exposition and correct a few mistakes.

\subsection*{Outline of the paper} 

Besides the introduction, there are 5 other sections in this paper. In Section~\ref{sec:prelim} we review some preliminaries, introduce several necessary definitions, and prove some auxiliary lemmas in a general framework. In Section~\ref{sec:stability} we introduce the relevant notions of (local) stability and make some first general observations. In Section~\ref{sec:amenable} we study amenable groups and prove Theorems~\ref{intro:thm:char} and \ref{intro:thm:amenable}. In Section~\ref{sec:examples} we prove Proposition~\ref{intro:prop:PtoHS} and use this together with our character criterion to establish Theorem~\ref{intro:thm:full}. Finally in Section~\ref{sec:T} we prove Theorem~\ref{intro:thm:T}.

\section{Preliminaries}\label{sec:prelim}

All abstract groups (i.e. those not coming with a metric structure) will be countable by default. Throughout, $\Gamma$ will denote a countable discrete group with identity element $e$. We denote by $\bM_n$ the algebra of complex $n\times n$ matrices, and by $\norm{\cdot}_2$ the normalized Hilbert--Schmidt norm on $\bM_n$, i.e. the norm induced by the normalized trace $\tau_n$. We will denote by $\U(n)$ the subgroup of $\bM_n$ consisting of unitary matrices, and by $\mathbbm{1}$ the identity matrix. Unless specified otherwise, the word ``representation'' means ``finite-dimensional unitary representation''. By convention, the set of natural numbers $\N$ starts at $1$.

\subsection{Sequences of metric groups}

We start with the object involved in most of the upcoming definitions:

\begin{definition}
\label{def:metric}

A \emph{family of metric groups} $\G$ is a family $\{ (G, \dG) \}$ where each $G$ is a group and $\dG$ is a bi-invariant metric on $G$. A countable family of metric groups indexed by $\N$, will be referred to as a \emph{sequence of metric groups} $\G = \{ (G_n, d_n) : n \in \N \}$.
\end{definition}

This definition is sufficient to define stability and local stability in the following section. In order to define the flexible notions, we will need the following additional structure:

\begin{definition}
\label{def:directed}

A \emph{directed sequence of metric groups} consists of the following data:
\begin{itemize}
    \item A sequence of metric groups $\G = \{(G_n,\dn) : n \in \N\}$;
    \item Compatible injective homomorphisms $\iota_{n, m} : G_n \to G_m$ for all $m \geq n$, such that $\iota_{n, n} = \id_{G_n}$;
    \item A function
    \[\Dn : G_n \times \bigsqcup_{m \geq n} G_m \to \R_{\geq 0}\]
    such that $\Dn$ restricts to the usual distance function $\dn$ on $G_n \times G_n$, and for all $\ell \geq m \geq n$ and all $g \in G_n, h \in G_m$:
    \[\Dn(g, \iota_{m, \ell}(h)) = \Dn(g, h).\]
\end{itemize}
\end{definition}

Note that the injective homomorphisms need not be isometric. The following will be the main example throughout this paper.

\begin{example}
\label{ex:HS}

Let $\cU \coloneqq \{ (\U(n), \norm{\cdot}_2) : n \in \N \}$, i.e. the unitary groups $\U(n)$ endowed with the distance induced by the normalized Hilbert--Schmidt norm $(M, N) \mapsto \norm{M - N}_2$. There are injective homomorphisms $\iota_{n, m} : \U(n) \to \U(m)$ for all $m \geq n$, given by inclusion in the upper-left corner. The distance function on $\U(n)$ extends as follows: for all $m \geq n$,
\[(M, N) \in \U(n) \times \U(m) \quad \mapsto \quad \norm{M - P_n N P_n}_2.\]
Here $P_n : \C^m \to \C^n$ is the projection onto the first $n$ coordinates. This makes $\cU$ into a directed sequence of metric groups.
\end{example}

We will also encounter the following example:

\begin{example}
\label{ex:symmetric}

Let $\cS \coloneqq \{ (\Sym(n), \dHamm) : n \in \N \}$ be the finite symmetric groups $\Sym(n)$ endowed with the normalized Hamming distance:
\[ \dHamm(\sigma, \tau) = \frac{1}{n} \#\{ i \in \{1, \ldots, n\} : \sigma(i) \neq \tau(i) \}. \]
There are injective homomorphisms $\iota_{n, m} : \Sym(n) \to \Sym(m)$ for all $m \geq n$, given by declaring all points in $\{ n+1, \ldots, m \}$ as fixed points of elements in $\iota_{n, m} (\Sym(n))$. The distance function on $\Sym(n)$ extends as follows: for all $m \geq n$,
\[
(\sigma, \tau) \in \Sym(n) \times \Sym(m) \quad \mapsto \quad \frac{1}{n} \# \{i \in \{1, \ldots, n\} : \sigma(i) \neq \tau(i) \}.
\]
This makes $\cS$ into a sequence of metric groups.
\end{example}

\begin{remark}
In their initial definition of flexible stability, Becker and Lubotzky \cite[Section 4.4]{BL:T} include an additive correction of $\frac{m-n}{n}$ in the distance formula above. When we formalize flexible stability, we will add the requirement that $m \sim n$ asymptotically, which makes this additive correction superfluous.
\end{remark}

The two families $\cU$ and $\cS$ are related, since one can view any permutation as its corresponding (unitary) permutation matrix. For future reference, we include the following easy fact.

\begin{lemma}
\label{lem:StoU}

Let $\Sigma_n \coloneqq \Sym(n) \to \U(n)$ be the homomorphism defined by the standard permutation representation. Then, for all $\sigma, \tau \in \Sym(n)$ it holds that
\[\| \Sigma_n(\sigma) - \Sigma_n(\tau) \|_2 = \sqrt{2 \dHamm(\sigma, \tau)}.\]
\end{lemma}

\subsection{Approximate homomorphisms}

\begin{definition}\label{def:partialapprox}
    Let $\cC$ be a class of groups. A sequence of maps $\vphi_n:\Gamma \to G_n \in \cC$ is called
    \begin{enumerate}[label = (\roman*)]
        \item a \emph{partial homomorphism} if for all $g,h\in\Gamma$, there exists $N\in\N$ such that for all $n\geq N$:
        \[
        \vphi_n(gh) = \vphi_n(g)\vphi_n(h).
        \]
    \end{enumerate}
    Assume now that $\G = \{(G_n,\dn) : n \in \N\}$ is a sequence of metric groups. Then $\vphi_n:\Gamma \to G_{k_n} \in \G$ is called
    \begin{enumerate}[label = (\roman*)]\setcounter{enumi}{1}
        \item an \emph{approximate homomorphism} (also sometimes \emph{almost homomorphism} or \emph{asymptotic homomorphism}) if for all $g,h\in\Gamma$:
        \[
        \dkn(\vphi_n(gh), \vphi_n(g)\vphi_n(h))\to 0.
        \]
        \item \emph{separating} if for all $e\neq g\in\Gamma$:
        \[
        \liminf\limits_{n \to \infty} \dkn(\vphi_n(g), 1_{G_{k_n}}) > 0.
        \]
    \end{enumerate}
    We will also use the shorthand notation $\vphi : \Gamma \to \G$ for the sequence of maps $\vphi_n$.
\end{definition}

Note that a partial homomorphism into a sequence of metric groups is in particular an approximate homomorphism. The question of whether every approximate homomorphism is a small perturbation of a partial homomorphism, is exactly the problem of local stability which we will introduce in Section~\ref{sec:stability}.

\begin{definition}
    Let $\G$ be a sequence of metric groups, and let $\Gamma$ be a group. We say that $\Gamma$ is \emph{$\G$-approximable} if there exists a separating approximate homomorphism $\vphi : \Gamma \to \G$.
\end{definition}

\begin{example}
\label{ex:sofic:hyperlinear}

    $\cU$-approximable groups are the so-called \emph{hyperlinear groups} \cite{radulescu}. They are also sometimes called \emph{Connes-embeddable groups}, since these are exactly the groups whose group von Neumann algebra embeds in an ultrapower of the hyperfinite II$_1$ factor. $\cS$-approximable groups are the so-called \emph{sofic groups} \cite{gromov, weiss}. It is known that all sofic groups are hyperlinear \cite{ES:amplification} (see also Lemma~\ref{lem:StoU}), and all amenable groups are sofic \cite{weiss}.
\end{example}

\begin{remark}
\label{rmk:amplification}

In the cases of $\cU$ and $\cS$, by taking appropriate tensor powers, one can show that approximability is equivalent to requiring the separating homomorphism to be ``maximally'' separating, i.e. where $\liminf_{n\to\infty} \dkn(\vphi_n(g), 1_{G_{k_n}})$ equals $\sqrt{2}$ and $1$ for $\cU$ and $\cS$ respectively, see \cite{ES:amplification}. Note however that in some other contexts, determining the equivalence between two such notions of approximability can be considerably harder \cite{linearsofic, linearsofic:gap}.
\end{remark}

Another ``approximation property'' with respect to a class of groups that will be relevant for our purposes is the following:

\begin{definition}
\label{def:LEF}
    Let $\mathcal{C}$ be a class of groups. A group $\Gamma$ is \emph{locally embeddable into $\mathcal{C}$} if for every finite subset $B \subset \Gamma$ there exist a group $C \in \mathcal{C}$ and a map $\vphi : \Gamma \to C$ such that $\vphi|_B$ is injective and $\vphi(gh) = \vphi(g)\vphi(h)$ for all $g, h \in B$.

    A group $\Gamma$ is called \emph{LEF} if it is locally embeddable in the class of finite groups.
\end{definition}

\begin{remark}\label{rk:partial:localemb}
    With the terminology introduced above, we note that a group $\Gamma$ is locally embeddable into a class of groups $\cC$ if and only if there exists a partial homomorphism $\vphi_n:\Gamma\to G_n\in\cC$ that is \emph{injective}, i.e. for all $g,h\in\Gamma$ there exists $N\in\N$ such that $\vphi_n(g)\neq\vphi_n(h)$ for $n\geq N$ (or equivalently, for all $e\ne g\in\Gamma$ we have $\vphi_n(g)\ne e$ for $n$ large enough).
\end{remark}

Note that a group is LEF if and only if it is locally embeddable into the class of finite symmetric groups, as every finite group admits an exact embedding into a finite symmetric group. Furthermore, LEF groups admit the following nice characterization in terms of separating partial homomorphisms with values in symmetric or unitary groups:

\begin{lemma}
\label{lem:LEF:hyperlinear}
    Let $\Gamma$ be a group. The following are equivalent:
    \begin{enumerate}
        \item $\Gamma$ is LEF;
        \item There exists a separating partial homomorphism $\vphi : \Gamma \to \cS$;
        \item There exists a separating partial homomorphism $\vphi : \Gamma \to \cU$;
        \item $\Gamma$ is locally embeddable into $\cU$.
    \end{enumerate}
\end{lemma}

\begin{proof}
$(1) \Rightarrow (2).$ Suppose that $\Gamma$ is LEF. Write $\Gamma$ as an increasing union of finite sets $B_n$, and let $\vphi_n : \Gamma \to F_n$ be as in Definition \ref{def:LEF} with respect to the finite set $B_n$, where each $F_n$ is a finite group of order $k_n$. We assume without loss of generality that $e \in B_n$ for every $n$. Compose $\vphi_n$ with the action of $F_n$ on itself by left multiplication. This yields a map $\phi_n : \Gamma \to \Sym(k_n)$, and these maps form a separating partial homomorphism.

$(2) \Rightarrow (3).$ It suffices to compose with the map $\Sigma_n : \Sym(n) \to \U(n)$ and use Lemma \ref{lem:StoU}.

$(3) \Rightarrow (4)$. This is clear from the definitions (cf. Remark~\ref{rk:partial:localemb}).

$(4) \Rightarrow (1).$ Suppose that $\Gamma$ locally embeds into $\cU$. Fix a finite set $B\subset \Gamma$. By assumption there exists $n\in \N$ and $\vphi:\Gamma\to \U(n)$ such that $\vphi(gh) = \vphi(g)\vphi(h)$ for all $g, h \in B$, and $\vphi|_B$ is injective. Let $G \leq \U(n)$ be the subgroup generated by $\vphi(B)$. By Mal'cev's Theorem \cite{malcev}, being a finitely generated linear group, $G$ is residually finite. Therefore, there exists a finite quotient $F$ of $G$ into which $\vphi(B)$ is mapped injectively.
We can map all the elements not in $G$ to $1_F$ and compose the resulting map with $\vphi$. This yields a map $\psi : \Gamma \to F$ witnessing the condition in Definition~\ref{def:LEF} for $B$. 
Hence $\Gamma$ is LEF.
\end{proof}

\begin{example}
    Examples of LEF groups include residually finite groups, and more generally MAP groups, topological full groups of minimal Cantor systems (see Section~\ref{sec:examples}), and they are stable under some natural group-theoretic constructions such as taking wreath products \cite{LEF}.

    Furthermore, we note that LEF is mainly a useful concept for non-finitely presented groups, since finitely presented LEF groups are automatically residually finite. This follows for instance from Lemma~\ref{lem:fp:partial}. In particular, non-examples of LEF groups include finitely presented but non-residually finite groups, such as the Baumslag--Solitar group $BS(2,3)$, the Higman group, or Thompson's groups $F, T$ and $V$. For an amenable such group, see e.g. \cite{baumslag:solvable}.
\end{example}

\subsection{Marked groups}

For the study of local stability, the language of marked groups will be useful, and we briefly review it here.

\begin{definition}
\label{def:marked}

    Let $\bF = \bF_S$ be a free group freely generated by a (countable) set $S$. An \emph{$\bF$-marked (or $S$-marked) group}, is a group $\Gamma$ together with a surjective homomorphism $\bF \to \Gamma$, called a \emph{marking}. Two $\bF$-marked groups are \emph{isomorphic} if there is an isomorphisms that intertwines the markings. Equivalently, an $\bF$-marked group is a group $\Gamma$ with $\abs{S}$ labeled elements that generate $\Gamma$.

    We denote by $\mathcal{G}_{\bF}$ the set of isomorphism classes of $\bF$-marked groups; when $\bF$ has rank $r$, we will simply write $\mathcal{G}_r$. Taking the kernel of a marking identifies $\mathcal{G}_{\bF}$ with the set of all normal subgroups of $\bF$, which in turn is endowed with the subspace topology of $2^{\bF}$. The \emph{space of $\bF$-marked groups} is the set $\mathcal{G}_{\bF}$ endowed with this topology.
\end{definition}

The space of marked groups was introduced by Grigorchuk in \cite{grigorchuk:marked}; we refer the reader to \cite{limit, cellular} for more modern introductions.
Since $\bF$ is countable, $2^{\bF}$ is homeomorphic to the Cantor space. $\mathcal{G}_{\bF}$ is identified with a closed subspace, therefore it is compact, Hausdorff and metrizable, and thus its topology is completely determined by convergence of sequences. Explicitly, a sequence of markings $\pi_n : \bF \to \Gamma_n$ converges to a marking $\pi : \bF \to \Gamma$ if and only if for every word $w\in\bF$:
\begin{align}\label{eq:markedconvergence}
\begin{split}
    &w \in \ker(\pi) \iff \exists N\in\N:\forall n\geq N: w \in \ker(\pi_n), \text{ and}\\
    &w \notin \ker(\pi) \iff \exists N\in\N:\forall n\geq N: w \notin \ker(\pi_n).
\end{split}
\end{align}
Furthermore, in this case, we have
\begin{equation}\label{eq:markedconvergence2}
 \ker(\pi)=\liminf \limits_n\ker(\pi_n):=\bigcup\limits_{k=1}^{\infty}\bigcap\limits_{i=k}^{\infty} \ker(\pi_n).
\end{equation}
If \eqref{eq:markedconvergence} holds, we say that $\Gamma$ is the \emph{marked limit} of the $\Gamma_n$. If $S$ is finite, then, denoting by $B_\bF(r)$ the ball of radius $r$ in the Cayley graph of $\bF$ with respect to $S$, convergence of $\pi_n$ to $\pi$ is equivalent to the existence of $r_n\to\infty$ such that
\[
B_\bF(r_n)\cap \ker(\pi_n) = B_\bF(r_n)\cap\ker(\pi).
\]

The notions of local embeddings and marked limits are connected via the following theorem:

\begin{theorem}[{\cite[Theorems 7.1.16, 7.1.19]{cellular}}]
\label{thm:LE:limit}
    Let $\bF \to \Gamma$ be a marked limit of $\bF \to \Gamma_n$. Then $\Gamma$ is locally embeddable into the sequence $\{\Gamma_n\}_{n \in \N}$. More precisely, there exist an increasing sequence of finite sets $S \cup \{ e \} \subset B_n \subset \Gamma$ exhausting $\Gamma$ and maps $\vphi_n : \Gamma \to \Gamma_n$ preserving the markings such that $\vphi_n(gh) = \vphi_n(g)\vphi_n(h)$ for all $g, h \in B_n$, and $\vphi_n|_{B_n}$ is injective.

    Conversely, if $\Gamma$ is locally embeddable into a class $\mathcal{C}$ which is closed under taking subgroups, then $\Gamma$ is a marked limit of groups belonging to $\mathcal{C}$. More precisely, if $(B_n)_n$ is an exhaustion of $\Gamma$ by finite sets, and $\vphi_n : \Gamma \to C_n \in \mathcal{C}$ is multiplicative and injective on $B_n$, then $\pi : \bF \to \Gamma$ is the limit of $(\pi_n : \bF \to \langle \vphi_n(S) \rangle)_n$, where $\pi_n$ is the unique homomorphism that coincides with $\vphi_n$ on S.
\end{theorem}

\begin{example}
\label{ex:infinitelypresented}

Let $\Gamma$ be a finitely generated group with a presentation $\langle S \mid R \rangle$ where $S$ is finite and $R = \{ r_1, r_2, \ldots \}$ is infinite. Let $\Gamma_n = \langle S \mid r_1, \ldots, r_n \rangle$. Then the sequence of canonical markings $\bF_S \to \Gamma_n$ converges to $\bF_S \to \Gamma$.
\end{example}

While the notion of LEF group is standard and well-studied, for our purposes linear groups will sometimes be more suitable. In this respect, the following corollary of Lemma~\ref{lem:LEF:hyperlinear} and Theorem~\ref{thm:LE:limit} is worth noting:

\begin{lemma}
\label{lem:marked:LEF}
    Let $\bF \to \Gamma$ be a marked group. The following are equivalent:
    \begin{enumerate}
        \item $\Gamma$ is LEF;
        \item $\Gamma$ is locally embeddable into the class of finite-dimensional unitary groups;
        \item $\bF \to \Gamma$ is a marked limit of finite groups;
        \item $\bF \to \Gamma$ is a marked limit of subgroups of finite-dimensional unitary groups.
    \end{enumerate}
\end{lemma}

\begin{remark}
\label{rem:convergence:quotient}
    For a marking $\pi\colon\mathbb{F}_S\to\Gamma$ we have the following. Let $\vphi_n\colon\Gamma\to G_n$ be a partial homomorphism, and denote by $\psi_n\colon \mathbb{F}_S\to G_n$ the corresponding homomorphisms defined by $\psi_n(s)=\vphi_n(\pi(s))$ for $s\in S$. By construction
    $\ker\pi \subseteq \liminf\limits_n \ker \psi_n.$
    Viewing $\psi_n$ as a homomorphism from $\Lambda_n\coloneqq \mathbb{F}_S/(\ker\pi\cap \ker\psi_n)$, we see that $\Gamma$ is the marked limit of $\pi_n\colon\mathbb{F}_S\to\Lambda_n$. Moreover, if $\langle\vphi_n(\pi(S))\rangle=G_n$ then the markings $\psi_n\colon \mathbb{F}_S\to G_n$ converge to $\mathbb{F}_S/\liminf\limits_n \ker \psi_n$, which is a quotient of $\Gamma$. 

    Conversely, let $\psi_n\colon\Lambda_n\to G_n$ be a sequence of homomorphisms, where $\pi_n\colon\mathbb{F}_S\to\Lambda_n$ are markings converging to $\Gamma$. Fix a map $q\colon \Gamma\to \mathbb{F}_S$ such that $\pi\circ q=\id_\Gamma$, and define $\varphi_n\coloneqq\psi_n\circ\pi_n\circ q\colon\Gamma\to G_n$.
    For any $g,h\in\mathbb{F}_S$ we have $q(g)q(h)=\alpha(g,h)q(gh)$ where $\alpha(g,h)\in\ker\pi$, hence $\vphi_n(g)\vphi_n(h)=\psi_n(\pi_n(\alpha(g,h)))\vphi_n(gh)$ and it follows that $\vphi_n$ is a partial homomorphism.  

    Note also that for any $w\in\mathbb{F}_S$, in both parts, we have $\psi_n(\pi_n(w))=\vphi_n(\pi(w))$ for $n$ large enough. 
\end{remark}

\subsection{Almost descending homomorphism}

Marked groups offer an alternative approach to the notion of approximate homomorphisms. For the remainder of this subsection, let $\pi : \bF \to \Gamma$ be a fixed marking.

\begin{definition}
\label{def:descending}
    Let $\G$ be a sequence of metric groups. A sequence of homomorphisms $f_n : \bF \to G_{k_n} \in \G$, denoted $f : \bF \to \G$, is said to be
    \begin{enumerate}[label = (\roman*)]
        \item an \emph{approximate homomorphism for $\Gamma$}, or \emph{descend to an approximate homomorphism of $\Gamma$}, if for all $w \in \ker(\pi)$:
        \[
        \dGkn(f_n(w), 1_{G_{k_n}})\to 0.
        \]
        \item a \emph{partial homomorphism for $\Gamma$}, or \emph{descend to a partial homomorphism of $\Gamma$}, if for all $w \in \ker(\pi)$ there exists $N \in \N$ such that for all $n \geq N$:
        \[ 
        f_n(w) = 1_{G_{k_n}}.
        \]
    \end{enumerate}
\end{definition}

The notions of approximate and partial homomorphisms on the level of $\Gamma$ and $\bF$ are related via the following, which justifies the terminology of \emph{descending} in the above definition.

\begin{lemma}
\label{lem:lift:descend}

    Let $\bF$ be a free group on the set $S$, and let $\pi:\bF\to \Gamma$ be a marking. Let $\G$ be a sequence of metric groups.
    \begin{enumerate}
        \item Let $\vphi : \Gamma \to \G$ be an approximate (respectively, partial) homomorphism, and for each $n$ let $f_n : \bF \to G_{k_n}$ be the homomorphism defined on the basis by $f_n(s) \coloneqq \vphi_n(\pi(s))$. Then $f : \bF \to \G$ is an approximate (respectively, partial) homomorphism for $\Gamma$.

        \item Let $f : \bF \to \G$ be an approximate (respectively, partial) homomorphism for $\Gamma$. Let $q: \Gamma \to \bF$ be a map such that $\pi q = \id_{\Gamma}$, and define $\vphi_n(g) \coloneqq f_n(q(g))$. Then $\vphi : \Gamma \to \G$ is an approximate (respectively, partial) homomorphism.
    \end{enumerate}
\end{lemma}

The proof is a straightforward calculation, see for instance \cite[Lemma 4.4]{ultrametric} for a detailed proof in a slightly different context.

The following easy lemma shows that for a finitely presented group, there are no genuine partial homomorphisms, in the sense that every partial homomorphism is ``close'' to a sequence of genuine homomorphisms. In particular, local stability and stability will coincide for finitely presented groups (see Lemma~\ref{lem:fp:local}).

\begin{lemma}
\label{lem:fp:partial}

Let $\Gamma$ be a finitely presented group, $\cG$ a sequence of metric groups, and let $\vphi_n:\Gamma\to G_{k_n}$ be a partial homomorphism. Then there exists a sequence of homomorphisms $\psi_n:\Gamma \to G_{k_n}$ such that for all $g\in\Gamma$, there exists $N\in\N$ such that for all $n\geq N$: $\vphi_n(g)=\psi_n(g)$.
\end{lemma}

\begin{proof}
    Let $\Gamma = \langle S \mid R \rangle$ be a finite presentation and let $\pi : \bF \to \Gamma$ be the corresponding marking. Let $f_n:\bF\to G_{k_n}$ be the corresponding partial homomorphism for $\Gamma$ from Lemma~\ref{lem:lift:descend}(1). Since $R \subset \ker(\pi)$ is finite, there exists $N$ such that for all $n \geq N$ we have $R \subset \ker(f_n)$. Hence $f_n$ induces a homomorphism $\psi_n : \Gamma \to G_n$ such that $\psi_n \circ \pi = f_n$ for $n\geq N$. Let $\psi_n$ be the trivial homomorphism for $n<N$.

    Since $\vphi$ is a partial homomorphism, there exist finite sets $B_n$ exhausting $\Gamma$ such that $\vphi_n(gh) = \vphi_n(g) \vphi_n(h)$ for all $g, h \in B_n$. Since for all $n \geq N$ the maps $\vphi_n, \psi_n : \Gamma \to G_n$ coincide on $S$ and are multiplicative on $B_n$, they coincide on all of $B_n$.
\end{proof}

\subsection{Ultraproducts}

Lastly, we add one further useful point of view. Let $\omega$ be a non-principal ultrafilter on $\N$. Let $(k_n)_n$ be a sequence of natural numbers, and let $\G$ be a sequence of metric groups. In the direct product $\prod_{n \in \N} G_{k_n}$, we consider the following subgroups:
\begin{align*}
    \cI_{alg} &\coloneqq \{ (g_n)_n \mid g_n \in G_{k_n}, \{ n \in \N \mid g_n = 1_{G_{k_n}} \} \in \omega \}; \\
    \cI_{met} &\coloneqq \{ (g_n)_n \mid g_n \in G_{k_n}, \lim\limits_{n \to \omega} \dGkn(g_n, 1_{G_{k_n}}) = 0 \}.
\end{align*}
Note that $\cI_{alg} \leq \cI_{met}$, and by bi-invariance and the triangle inequality, both are normal subgroups of the direct product.

\begin{definition}
\label{def:ultraproduct}

The group
\[\prod^{alg}_{\omega} G_{k_n} \coloneqq \prod\limits_{n \in \N} G_{k_n} / \cI_{alg} \]
is called the \emph{algebraic ultraproduct} of the sequence $(G_{k_n})_n$. The group
\[\prod^{met}_{\omega} G_{k_n} \coloneqq \prod\limits_{n \in \N} G_{k_n} / \cI_{met} \]
is called the \emph{metric ultraproduct} of the sequence $(G_{k_n})_n$.
\end{definition}

Note that there are quotient maps
\[
\prod\limits_{n \in \N} G_{k_n} \to \prod_\omega^{alg} G_{k_n} \to \prod_{\omega}^{met} G_{k_n}.
\]
By considering a set-theoretic section $q:\prod_{\omega}^{alg} G_{k_n}\to \prod_n G_{k_n}$, we see that each homomorphism $\vphi_{alg}\colon\Gamma\to \prod_{\omega}^{alg} G_{k_n}$ can be written as a sequence of maps $\vphi_n\coloneqq (q\vphi_{alg})_n \colon \Gamma\to G_{k_n}$. The same holds for the metric ultraproduct.

The notions of approximate and partial homomorphisms have an immediate reformulation in the ultraproduct framework:

\begin{lemma}
\label{lem:approximate:ultraproduct}

Let $\G$ be a sequence of metric groups and let $\vphi_n : \Gamma \to G_{k_n}$ be maps. Denote by $\vphi_\infty = (\vphi_n)_n: \Gamma \to \prod_{n \in \N} G_{k_n}$ the corresponding map to the direct product.
\begin{enumerate}
    \item If $\vphi_n : \Gamma \to G_{k_n}$ is a (separating) approximate homomorphism, then $\vphi_\infty$ descends to an (injective) homomorphism $\vphi_{met} : \Gamma \to \prod_{\omega}^{met} G_{k_n}$. Conversely, if $\vphi_\infty$ descends to an (injective) homomorphism $\vphi_{met} : \Gamma \to \prod_{\omega}^{met} G_{k_n}$, then there is a strictly increasing sequence $(n_m)_m \in \omega$ such that $(\vphi_{n_m})_m$ is a (separating) approximate homomorphism.
    \item If $\vphi_n : \Gamma \to G_{k_n}$ is a partial homomorphism, then $\vphi_\infty$ descends to a homomorphism $\vphi_{alg} : \Gamma \to \prod_{\omega}^{alg} G_{k_n}$. Conversely, if $\vphi_\infty$ descends to an homomorphism $\vphi_{alg} : \Gamma \to \prod_{\omega}^{alg} G_{k_n}$, then there is a strictly increasing sequence $(n_m)_m \in \omega$ such that $(\vphi_{n_m})_m$ is a partial homomorphism.
\end{enumerate}
\end{lemma}

This follows directly from the definitions, we refer the reader to \cite[Sections 7.2 and 7.6]{cellular} for proofs in the context of the family $\cS$, which carry through verbatim in the general context.

\subsection{Characters}

For amenable groups, we will be able to establish a \emph{local character criterion} in Theorem~\ref{thm:char} for checking local HS-stability (see Definition~\ref{def:HSstable}). Therefore we briefly review characters here. Recall that a function $\chi:\Gamma\to\C$ on a countable discrete group $\Gamma$ is called \emph{positive definite} if for every $n\in\N$ and every set $\{g_1, \ldots, g_n\}\subset\Gamma$, the matrix $[\chi(g_i^{-1}g_j)]_{i,j=1}^n$ is positive definite, or equivalently, if $\chi$ induces a state on the group algebra $\C\Gamma$.

\begin{definition}
    Let $\Gamma$ be a countable discrete group. A map $\chi:\Gamma\to\C$ is called a \emph{character} if
    \begin{enumerate}
        \item $\chi$ is positive definite,
        \item $\chi(gh) = \chi(hg)$ for all $g,h\in\Gamma$, and
        \item $\chi(e) = 1$.
    \end{enumerate}
    A character is called \emph{extremal} (or \emph{indecomposable}) if it cannot be written as a nontrivial convex combination of characters.
\end{definition}

We note that with the topology of pointwise convergence, the space of characters is a compact convex space, and the extremal characters are exactly the extreme points of this compact convex set.
Hence by the Krein--Milman Theorem, the space of characters is equal to the closed convex hull of the set of extremal characters.

\begin{example}
    If $\pi:\Gamma\to (M,\tau)$ is a representation of $\Gamma$ into a tracial von Neumann algebra (e.g. a finite-dimensional representation), then $\tau\circ\pi$ is a character on $\Gamma$. The GNS construction \cite[1.B]{characters} shows that every character arises in this way. Moreover, it is well-known and easy to check that extremal characters correspond to exactly those representations where the von Neumann algebra generated by $\pi(\Gamma)$ is a factor.
\end{example}

\begin{example} 
\label{ex:trivialextension}
   Suppose $\Sigma\trianglelefteq\Gamma$ is a normal subgroup, and $\chi:\Sigma\to\C$ is a $\Gamma$-invariant character. Then putting $\chi(g)=0$ for $g\in \Gamma\setminus\Sigma$ defines a character on $\Gamma$, which is sometimes called the \emph{trivial extension} of $\chi$ from $\Sigma$ to $\Gamma$. For instance, this works for any character defined on a central subgroup.
   More generally, one can induce sufficiently invariant characters on sufficiently normal subgroups, see for instance \cite[Section~3]{levit:vigdorovich}.
\end{example}

\section{Stability}\label{sec:stability}

\subsection{Definitions}

Notions of stability can be traced back to early works of von Neumnann \cite{vN} and Turing \cite{turing}, and more concretely to Ulam \cite[Chapter VI]{ulam}. The modern definition was introduced by Arzhantseva and P\u{a}unescu \cite{AP:commuting} in the case of $\G = \cS$, and the case of $\G = \cU$ was formally introduced by Hadwin and Shulman in \cite{HS2}. The flexible notions were suggested in \cite{BL:T}, and the local notion in the case of $\G = \cS$ was introduced in \cite{bradford}.

In order to define stability, the first step is to formalize what we mean by ``small perturbation''. Recall the notion of (directed) sequence of metric groups from Definitions \ref{def:metric} and \ref{def:directed}.

\begin{definition}
\label{def:close}

Let $\G$ be a sequence of metric groups, and let $\vphi_n, \psi_n : \Gamma \to G_{k_n}$ be approximate homomorphisms. We say that $(\vphi_n)_n$ and $(\psi_n)_n$ are \emph{close} if for all $g \in \Gamma$ it holds that
\[
\lim\limits_{n \to \infty} \dGkn(\vphi_n(g), \psi_n(g)) = 0.
\]

If $\G$ is moreover a directed sequence of metric groups, and $\vphi_n : \Gamma \to G_{k_n}, \psi_n : \Gamma \to G_{m_n}$ are approximate homomorphisms, where $m_n \geq k_n$. We say that $(\vphi_n)_n$ and $(\psi_n)_n$ are \emph{close} if for all $g \in \Gamma$ it holds
\[\lim\limits_{n \to \infty} \DGkn(\vphi_n(g), \psi_n(g)) = 0.\]
\end{definition}

The second half of the definition says that, in case $\G$ is a directed sequence of metric groups, we can use the additional structure to compare two approximate homomorphisms with potentially different range. Recall that the function $\DGkn : G_{k_n} \times G_{m_n} \to \R_{\geq 0}$ restricts to the usual distance function on $G_{k_n} \times G_{k_n}$, so the two definitions are compatible.

A first observation is that it suffices to check closeness of $\vphi_n, \psi_n : \Gamma \to G_{k_n}$ on a generating set.

\begin{lemma}
\label{lem:close:generatingset}

Let $\G$ be a sequence of metric groups, $\vphi_n, \psi_n : \Gamma \to G_{k_n} \in \G$ approximate homomorphisms, and suppose $\Gamma$ is generated by a set $S$. Then $\vphi_n$ and $\psi_n$ are close if and only if
\[\lim\limits_{n \to \infty} d_{G_{k_n}}(\vphi_n(s), \psi_n(s)) = 0\]
for all $s \in S$.
\end{lemma}

\begin{proof}
The ``only if'' is obvious, the ``if'' follows from the triangle inequality and the definition of approximate homomorphism.
\end{proof}

\begin{definition}
\label{def:stable}
    Let $\G$ be a sequence of metric groups. The group $\Gamma$ is called
    \begin{enumerate}[label = (\roman*)]
        \item \emph{$\G$-stable} if every approximate homomorphism $\vphi_n : \Gamma \to G_{k_n}$ is close to a sequence of genuine homomorphisms $\psi_n : \Gamma \to G_{k_n}$.
        \item \emph{locally $\G$-stable} if every approximate homomorphism $\vphi_n : \Gamma \to G_{k_n}$ is close to a partial homomorphism $\psi_n : \Gamma \to G_{k_n}$.
    \end{enumerate}
\end{definition}

In case $\G$ is a directed sequence of metric groups, we can also define flexible notions of stability:

\begin{definition}
\label{def:flexibly}
    Let $\G$ be a directed sequence of metric groups. The group $\Gamma$ is called
    \begin{enumerate}[label = (\roman*)]
        \item \emph{flexibly $\G$-stable} if for every approximate homomorphism $\vphi_n : \Gamma \to G_{k_n}$ there exists a sequence $m_n \geq k_n$ with $\frac{m_n}{k_n} \xrightarrow{n \to \infty} 1$, and homomorphisms $\psi_n : \Gamma \to G_{m_n}$ such that $\vphi_n : \Gamma \to G_{k_n}$ and $\psi_n : \Gamma \to G_{m_n}$ are close.
        \item \emph{very flexibly $\G$-stable} if in (i) we do not require that $\frac{m_n}{k_n} \xrightarrow{n \to \infty} 1$.
        \item \emph{flexibly locally $\G$-stable} if for every approximate homomorphism $\vphi_n : \Gamma \to G_{k_n}$ there exists a sequence $m_n \geq k_n$ with $\frac{m_n}{k_n} \xrightarrow{n \to \infty} 1$, and a partial homomorphism $\psi_n : \Gamma \to G_{m_n}$ such that $\vphi_n : \Gamma \to G_{k_n}$ and $\psi_n : \Gamma \to G_{m_n}$ are close.
        \item \emph{very flexibly locally $\G$-stable} if in (iii) we do not require the condition $\frac{m_n}{k_n} \xrightarrow{n \to \infty} 1$.
    \end{enumerate}
\end{definition}

In our two main cases of interest, we rephrase the definitions by analogy with existing terminology:

\begin{definition}
\label{def:HSstable}

A ((very) flexibly) (locally) $\cS$-stable group will be called \emph{((very) flexibly) (locally) P-stable}. A ((very) flexibly) (locally) $\cU$-stable group will be called \emph{((very) flexibly) (locally) HS-stable}.
\end{definition}

\subsection{Reformulations}

The notions of stability admit reformulations in terms of markings and ultraproducts, which we work out next.

\begin{definition}
\label{def:close:descend}

Let $\pi : \bF \to \Gamma$ be a marking, and let $\G$ be a sequence of metric groups. Let $f_n, h_n : \bF \to G_{k_n}$ be two sequences descending to approximate homomorphisms for $\Gamma$. We say that $f$ and $h$ are \emph{close} if for all $w \in \bF$, it holds that
\[\lim\limits_{n \to \infty} \dGkn(f_n(w), h_n(w)) = 0.\]
Just as in Definition \ref{def:close}, this definition extends to approximate homomorphisms for $\Gamma$ with potentially different ranges, in case $\G$ is a directed sequence of metric groups.
\end{definition}

\begin{remark}
\label{rem:close:basis}

Let $S$ be a generating set of $\bF$. An easy application of the triangle inequality implies that $f$ and $h$ are close if and only if
\[\lim\limits_{n \to \infty} \dGkn(f_n(s), h_n(s)) = 0\]
for all $s \in S$.
\end{remark}

The proof of the following lemma is a straightforward calculation using Lemma~\ref{lem:lift:descend} which we leave up to the interested reader.

\begin{lemma}
\label{lem:stable:descend}

Let $\G$ be a sequence of metric groups, and let $\pi : \bF \to \Gamma$ be a marking.
\begin{enumerate}
    \item $\Gamma$ is $\G$-stable if and only if every sequence of homomorphism $f : \bF \to G_{k_n}$ which is an approximate homomorphism for $\Gamma$ is close to a sequence of homomorphisms $h : \bF \to G_{k_n}$ that are pullbacks of genuine homomorphisms of $\Gamma$.
    \item $\Gamma$ is locally $\G$-stable if and only if every sequence of homomorphism $f : \bF \to G_{k_n}$ which is an approximate homomorphism for $\Gamma$ is close to a sequence of homomorphisms $p : \bF \to G_{k_n}$ which is a partial homomorphism for $\Gamma$.
\end{enumerate}
\end{lemma}

Finally, one can conveniently describe (local) stability in terms of ultraproducts. Fix a non-principal ultrafilter $\omega$ on $\N$, and suppose $\vphi_n, \psi_n : \Gamma \to G_{k_n}$ are approximate homomorphisms. It is immediate from the definitions that if they are close, then the corresponding maps $\vphi_{met}, \psi_{met} : \Gamma \to \prod_{\omega}^{met} G_{k_n}$ coincide (cf. Lemma \ref{lem:approximate:ultraproduct}). Conversely, if $\vphi_{met}$ and $\psi_{met}$ coincide, then there exist an increasing sequence $n_m$ such that $(\vphi_{n_m})_{m}$ and $(\psi_{n_m})_{m}$ are close.

The following lemma gives an interpretation of stability in terms of lifting properties. This was first noticed in \cite{AP:commuting} for stability, see \cite[Section 2.3.3]{bradford} for the analogous statement for local stability.

\begin{lemma}
\label{lem:stable:ultraproduct}
Let $\G$ be a sequence of metric groups, and let $\Gamma$ be a group. Let $\omega$ be a non-principal ultrafilter on $\N$. Then
\begin{enumerate}
    \item $\Gamma$ is $\G$-stable if and only if for every sequence of natural numbers $(k_n)_n$, every homomorphism $\Gamma \to \prod_\omega^{met} G_{k_n}$ lifts to a homomomorphism $\Gamma \to \prod_{n \in \N} G_{k_n}$.
    \item $\Gamma$ is locally $\G$-stable if and only if for every sequence of natural numbers $(k_n)_n$, every homomorphism $\Gamma \to \prod_\omega^{met} G_{k_n}$ lifts to a homomomorphism $\Gamma \to \prod_\omega^{alg} G_{k_n}$.
\end{enumerate}
\end{lemma}

We summarize the lemma in the following diagram:

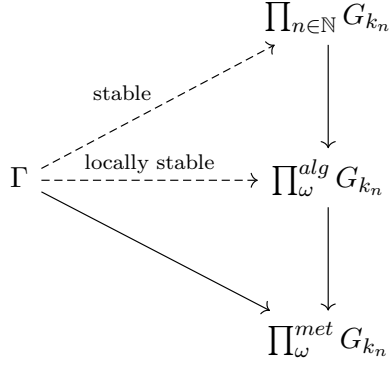
\[\begin{tikzcd}
	&&& {\prod_{n \in \N} G_{k_n}} \\
	\\
	\Gamma &&& {\prod_{\omega}^{alg} G_{k_n}} \\
	\\
	&&& {\prod_{\omega}^{met} G_{k_n}}
	\arrow[from=3-1, to=5-4]
	\arrow["{\text{locally stable}}", dashed, from=3-1, to=3-4]
	\arrow["{\text{stable}}", dashed, from=3-1, to=1-4]
	\arrow[from=1-4, to=3-4]
	\arrow[from=3-4, to=5-4]
\end{tikzcd}\]

\subsection{First basic observations}

The following lemma is clear from, e.g. the previous commutative diagram:

\begin{lemma}
Let $\G$ be a sequence of metric groups. If $\Gamma$ is $\G$-stable, then it is locally $\G$-stable.
\end{lemma}

Next, we note that although local $\G$-stability is weaker than $\G$-stability (we will see our first examples in the next subsection), there is no difference between the two notions for finitely presented groups (cf. \cite[Lemma~2.14]{bradford}).

\begin{lemma}\label{lem:fp:local}

    Let $\G$ be a sequence of metric groups, and assume $\Gamma$ is finitely presented. Then $\Gamma$ is locally $\G$-stable if and only if $\Gamma$ is $\G$-stable. The same statement holds for the (very) flexible versions.
\end{lemma}

\begin{proof}
This follows directly from Lemma \ref{lem:fp:partial}.
\end{proof}

Now let us specialize to the sequence $\cU$.
It is immediate that a ((very) flexibly) HS-stable hyperlinear group $\Gamma$ is MAP, where we recall that a group is called MAP (``maximally almost periodic'') if finite-dimensional representations separate points. If we furthermore assume that $\Gamma$ is finitely generated, then by Mal'cev's Theorem \cite{malcev}, it is necessarily residually finite. A first observation is that a similar fact holds for locally HS-stable hyperlinear groups (cf. \cite[Lemma~2.15]{bradford}), where in contrast to the above, the assumption of finite generation is not necessary in the local setting.

\begin{lemma}
\label{lem:stable:LEF}

    Assume $\Gamma$ is hyperlinear and very flexibly locally HS-stable. Then $\Gamma$ is LEF.
\end{lemma}

\begin{proof}
Since $\Gamma$ is hyperlinear, it admits a separating approximate homomorphism $\vphi_n :\Gamma\to \U(k_n)$. For $e \neq g \in \Gamma$, let $\rho_g \coloneqq \liminf \| \vphi_n(g) - I_{k_n} \|_2$, which is positive since $\vphi_n$ is separating. By very flexible local HS-stability, this approximate homomorphism is close to a partial homomorphism $\psi_n : \Gamma \to \U(m_n)$, where $m_n \geq k_n$. Now for $g \neq h \in \Gamma$, we have
\begin{align*}
    \liminf_n \| P_{k_n} \psi_n(g) P_{k_n} - P_{k_n} \psi_n(h) P_{k_n} \|_2 &= \liminf_n \| \vphi_n(g) - \vphi_n(h) \|_2 \\
    &= \liminf_n \| \vphi_n(g)\vphi_n(h)^{-1} - I_{k_n} \|_2 \\
    &= \liminf_n \| \vphi_n(gh^{-1}) - I_{k_n} \|_2 = \rho_{gh^{-1}} > 0.
\end{align*}
It follows that, for every finite set $B \subset \Gamma$, there exists $N$ such that for all $n \geq N$ the restriction $\psi_n|_B$ is injective. Thus $(\psi_n)_n$ is a local embedding of $\Gamma$ into $\cU$, and we conclude by Lemma \ref{lem:LEF:hyperlinear}.
\end{proof}

\begin{remark}
    The global properties MAP and residual finiteness coincide for finitely generated groups, but are distinct in general. On the other hand, the corresponding local properties are the same as LEF in both cases. This is because finite dimensional unitary groups are themselves LEF. It follows that both sofic locally P-stable groups and hyperlinear locally HS-stable groups are LEF, whereas hyperlinear HS-stable groups are MAP but not necessarily residually finite.
\end{remark}

\subsection{Relative stability}

Given an inclusion of groups $\Sigma\leq \Gamma$, one can also consider a relative version of (local) stability.

\begin{definition}
    Let $\Sigma\leq\Gamma$ be countable discrete groups, and suppose $\cG$ is a sequence of metric groups. We say that $\Sigma$ is \emph{(locally) $\cG$-stable relative to $\Gamma$} if for every approximate homomorphism $\vphi_n:\Gamma\to G_{k_n}$, the restriction $(\vphi_n|_{\Sigma})_n$ is close to a sequence $(\psi_n)_n$ of (partial) homomorphisms of $\Sigma$.
\end{definition}

\begin{example}
    If $\Gamma$ is (locally) $\cG$-stable, then every subgroup $\Sigma\leq\Gamma$ is (locally) $\cG$-stable relative to $\Gamma$.
\end{example}

\subsection{Directed unions}

A useful feature unique to the local setting is that local stability is clearly preserved under directed unions. This allows us to right away give some elementary examples of locally stable groups that are not stable. More generally, this is true for directed unions of relatively stable groups.

\begin{proposition}
\label{prop:directedunion}

Let $\Gamma$ be a countable group written as a directed union of subgroups $(\Gamma_i)_{i \in \N}$. Let $\G$ be a sequence of metric groups, and suppose that each $\Gamma_i$ is locally $\G$-stable relative to $\Gamma$. Then $\Gamma$ is locally $\cG$-stable.

If $\G$ is a directed sequence of metric groups, then the analogous result holds for (very) flexible local stability.
\end{proposition}

The proof is straightfoward, but we include it as it clearly encapsulates the \emph{local} nature of local stability, as opposed to the \emph{global} nature of ordinary stability.

\begin{proof}
Since $\Gamma$ is the directed union of the $\Gamma_i$, there exist finite sets $B_i \subset \Gamma_i$ such that $\Gamma$ is also the directed union of $B_i$. Let $\vphi_n : \Gamma \to G_{k_n}$ be an approximate homomorphism of $\Gamma$, and consider for each $i \in I$, the restriction $\vphi_n|_{\Gamma_i} \coloneqq \vphi_n^i : \Gamma_i \to G_{k_n}$, which is an approximate homomorphism of $\Gamma_i$. By local stability relative to $\Gamma$, there exists a partial homomorphism $\psi^i_n : \Gamma_i \to G_{k_n}$ of $\Gamma_i$ that is close to $\vphi^i_n$. In particular, for all $i \in I$ there exists $n(i)$ such that for all $n \geq n(i)$ the map $\psi^i_n : \Gamma_i \to G_{k_n}$ is multiplicative on $B_i$, and $d_{G_{k_n}}(\vphi^i_n(g), \psi^i_n(g)) \leq 1/i$ for all $g \in B_i$. We may choose $n(i)$ so that moreover $n(i) > n(i-1)$, and we set $i(n)$ to be the smallest integer such that $n \geq n(i(n))$; notice that by our choices $i(n) \to \infty$ as $n \to \infty$. Set $\psi_n : \Gamma \to G_{k_n}$ to be equal to $\psi^{i(n)}_n$ on $\Gamma_{i(n)}$, and to be trivial elsewhere. Then $\psi_n$ is multiplicative on $B_{i(n)}$, and $d_{G_{k_n}}(\vphi_n(g), \psi_n(g)) \leq 1/n$ for all $g \in B_{i(n)}$. Since $\Gamma$ is the directed union of the sets $B_i$, and $i(n) \to \infty$, we conclude that $(\psi_n)_n$ is a partial homomorphism, and that it is close to $(\vphi_n)_n$.

The proof in the (very) flexible case is completely analogous.
\end{proof}

\begin{corollary}
    Let $\G$ be a sequence of metric groups. If $\Gamma$ is a directed union of locally $\cG$-stable groups, then $\Gamma$ is locally $\cG$-stable.
\end{corollary}

\begin{corollary}
\label{cor:locallyvirtuallynilpotent}

Let $\Gamma$ be a countable group that is locally virtually nilpotent. Then $\Gamma$ is locally HS-stable and locally P-stable. In particular, this holds for all countable locally finite groups, and all countable abelian groups.
\end{corollary}

\begin{proof}
This is a direct consequence of Proposition \ref{prop:directedunion} together with \cite[Corollary I]{levit:vigdorovich} in the Hilbert--Schmidt case, and \cite[Theorem 1.2(1)]{BLT:IRS} in the permutation case.
\end{proof}

\begin{example}
\label{ex:unstable:locallystable}

The groups $S_\infty$ and $A_\infty$ of finitely supported (even) permutations of a countable set are countable and locally finite, thus both locally HS-stable and locally P-stable. However, both are amenable and not MAP, therefore not residually finite, so they cannot be HS-stable or P-stable, by the paragraph preceding Lemma \ref{lem:stable:LEF} (see also e.g. \cite[Theorem 7.2(ii)]{AP:commuting}).

To see that they are not MAP, consider a representation $\pi : A_\infty \to \U(n)$. By a theorem of Jordan (see e.g. \cite[Theorem 36.13]{jordan}), there exists a function $\iota : \N \to \N$ such that every finite subgroup of $\U(n)$ admits an abelian normal subgroup of index at most $\iota(n)$. Now $A_\infty$ is the directed union of the finite simple groups $A_k$, and it follows that $\pi(A_k)$ must be trivial for $k$ large enough. Since $A_\infty$ is simple, $\pi$ is the trivial representation. As a consequence, we also obtain that every representation $S_\infty \to \U(n)$ must factor through the abelianization $\Z / 2 \Z$.
\end{example}

\begin{example}
\label{ex:uncountablymany}

Let $p$ be an odd prime, and let $\Gamma_p := \mathrm{PSL}_2(\overline{\mathbb{F}_p})$, where $\overline{\mathbb{F}_p}$ is the algebraic closure of the finite field of $p$ elements $\mathbb{F}_p$, equivalently the directed union of the fields $\mathbb{F}_q$, where $q$ runs over all powers of $p$. Then $\Gamma_p$ is countably infinite, locally finite (since it is the directed union of the finite simple groups $\mathrm{PSL}_2(\mathbb{F}_q)$), and simple \cite[Corollary 8.14]{rotman}. By a similar reasoning as in the previous example, each $\Gamma_p$ cannot be HS-stable or P-stable, but it is both locally HS-stable and locally P-stable. For any nonempty set $\Pi$ of primes, the same conclusions hold for the group $\Gamma_\Pi := \bigoplus_{p \in \Pi} \Gamma_p$.

Next, we note that $\Gamma_p$ does not embed as a subgroup of $\Gamma_{\ell}$ for primes $p \neq \ell$. On the one hand $\Gamma_p$ contains subgroups of the form $(\mathbb{Z}/p \mathbb{Z})^k$ for every $k \in \N$, for instance among upper triangular matrices. On the other hand the finite abelian $p$-subgroups of $\Gamma_{\ell}$ are diagonalizable \cite[Lemma 7.4]{wehrfritz}, so contained in $\mathbb{F}_{\ell^j}^\times$ for some $j \in \N$, and thus cyclic. We can use this to show that the groups $\Gamma_\Pi$ are pairwise non-isomorphic. Indeed, suppose that there exists an isomorphism $f : \Gamma_\Pi \to \Gamma_{\Pi'}$ for two non-empty sets of primes $\Pi, \Pi'$. For each $p \in \Pi$ and each $p' \in \Pi'$ consider the composition
\[\begin{tikzcd}
	{\Gamma_p} & {\Gamma_{\Pi}} & {\Gamma_{\Pi'}} & {\Gamma_{p'}}
	\arrow[hook, from=1-1, to=1-2]
	\arrow["f", from=1-2, to=1-3]
	\arrow[two heads, from=1-3, to=1-4]
\end{tikzcd}\]
If $p \neq p'$, then this map cannot be injective, and since $\Gamma_p$ is simple it must be trivial. But $f$ is an isomorphism, and the quotients $\Gamma_{\Pi'} \twoheadrightarrow \Gamma_{p'} : p' \in \Pi'$ separate points, therefore at least one of these maps must be non-trivial, and we conclude that $p \in \Pi'$ too. Applying the same argument to the isomorphism $f^{-1} : \Gamma_{\Pi'} \to \Gamma_{\Pi}$ we obtain $\Pi = \Pi'$.

We have thus exhibited uncountably many non-isomorphic countable groups that are locally HS-stable and locally P-stable but neither HS-stable nor P-stable. For HS-stability, we will promote this to finitely generated examples in Theorem \ref{thm:full} (Theorem \ref{intro:thm:full}); for P-stability, this was already achieved by Bradford \cite{bradford}.
\end{example}

\section{Local HS-stability for amenable groups}\label{sec:amenable}

\subsection{The local character criterion}

We start this section by defining a ``local'' version of the character criterion from \cite[Theorem~4]{HS2} for amenable groups. We will then prove some basic properties related to this criterion, and establish in Section~\ref{sec:char:amenable} that for amenable groups, this criterion characterizes local HS-stability.

\begin{definition}[The local character criterion]\label{criterion}
    Let $\Gamma$ be a group generated by a countable (possibly finite) set $S$, and let $\pi:\bF_S\to \Gamma$ be the corresponding marking. We say that $\Gamma$ satisfies \emph{the local character criterion} if for every character $\chi:\Gamma\to\C$, there exists a sequence $(\Lambda_n)_n$ of $S$-marked groups converging to $\Gamma$ in the space of $S$-marked groups, together with finite-dimensional representations $\theta_n:\Lambda_n\to \U(k_n)$ such that $\tau_{k_n}\circ \theta_n$ converges pointwise to $\chi$ when all are viewed as characters of $\bF_S$. 
\end{definition}

We note the following useful equivalent formulation.

\begin{lemma}
\label{rem:equivalent:criterion}
 A group $\Gamma$ satisfies the local character criterion if and only if it satisfies the following property: 
 for every character $\chi\colon \Gamma\to\C$, there exists a partial homomorphism $\psi_n\colon \Gamma\to \U(k_n)$ such that $\tau_{k_n}\circ\psi_n$ converges pointwise to $\chi$. 
\end{lemma}
\begin{proof}
This follows immediately from Remark~\ref{rem:convergence:quotient}.
\end{proof}

The following lemma is the local counterpart for groups of \cite[Lemma~3.7]{HS1}; see also \cite[Proposition 6.1]{AP:commuting}.

\begin{lemma}\label{lem:dims}
    Let $\Gamma$ be a group generated by a countable set $S$, and let $\pi:\bF_S\to \Gamma$ be the corresponding marking. If $\Gamma$ satisfies the local character criterion~\ref{criterion}, then it satisfies the local character criterion with $k_n=n$.
\end{lemma}
\begin{proof}
We enumerate $\bF_S = \{g_1,g_2,g_3,\ldots\}$, and assume $\Gamma$ satisfies the local character criterion. Let $\chi:\Gamma\to\C$ be a character, and fix $N\in\N$. Write $B_N = \{g_1, \ldots, g_N\}$. By assumption, we can find a marking $\pi_N:\bF_S\to \Lambda_N$ satisfying $B_N\cap \ker(\pi_N) = B_N\cap \ker(\pi)$, and a representation $\theta_N:\Lambda_N\to \U(k_N)$ such that for all $1\leq i\leq N$:
\[
\abs{\tau_{k_N}(\theta_N(\pi_N(g_i))) - \chi(\pi(g_i))} < \frac{1}{N}.
\]
For $n\geq k_N$, write $n=a\cdot k_N + b$ with $0\leq b < k_N$, and consider the representation
\[
\theta_{N,n}\coloneqq \oplus_{i=1}^a \theta_N \,\bigoplus\,\oplus_{j=1}^b 1:\Lambda_N\to \U(n),
\]
where $1$ denotes the trivial representation. Then for $1\leq i\leq N$:
\[
\tau_n(\theta_{N,n}(\pi_N(g_i))) = \frac{a\cdot k_N\cdot\tau_{k_N}(\theta_N(\pi_N(g_i))) + b}{n}.
\]
Choose $l_N\in\N$ such that for $n\geq l_N$:
\[
\abs{\tau_n(\theta_{N,n}(\pi_N(g_i))) - \tau_{k_N}(\theta_N(\pi_N(g_i)))} < \frac{1}{N}.
\]
Making sure to choose $l_{N+1}>l_N$, we now define for $n\in\N$: 
\[
\Lambda_n \coloneqq \Lambda_N, \quad\text{and}\quad \sigma_n \coloneqq \theta_{N,n}, \quad\text{if}\quad l_N\leq n < l_{N+1}.
\]
It is then easy to check that $\Lambda_n$ and $\sigma_n$ satisfy the conditions from the local character criterion.
\end{proof}

Our next goal is to establish some basic facts related to the local character criterion~\ref{criterion}. 
For any group $\Gamma$, we denote by $\delta_e$ the character given by the indicator function of the identity, which we call the \emph{regular} character.

\begin{lemma}
\label{lem:regularcharacter}

    Let $\Gamma$ be a countable discrete group. Then $\Gamma$ is LEF if and only if the regular character $\delta_e$ satisfies the conclusion of the local character criterion.
\end{lemma}

\begin{proof}
We fix a generating set $S$ for $\Gamma$ and consider the corresponding marking $\pi : \bF_S \to \Gamma$. 

First, assume $\Gamma$ is LEF and choose a sequence of $S$-marked finite groups $F_n$ converging to $\Gamma$ in the space of $S$-marked groups. Let $\psi_n:\Gamma\to F_n$ be a partial homomorphism witnessing this, and denote by $\delta_e^{(n)}$ the regular character of $F_n$. Then $\delta_e^{(n)}\circ \psi_n$ converges pointwise to $\delta_e$ when all are viewed as characters of $\bF_S$, hence $\delta_e$ satisfies the conclusion of the local character criterion.

Conversely, assume $\delta_e$ satisfies the conclusion of the local character criterion. Choose a sequence of $S$-marked groups $\pi_n:\bF_S\to\Lambda_n$ converging to $\pi : \bF_S \to \Gamma$, together with finite-dimensional representations $\theta_n:\Lambda_n\to \U(k_n)$ such that $\tau_{k_n}\circ \theta_n$ converges to $\delta_e$ pointwise as characters of $\bF_S$. Following Theorem~\ref{thm:LE:limit}, choose a partial homomorphism $\psi_n:\Gamma\to \Lambda_n$ corresponding to $(\pi_n)_n$. Let $e\neq g\in\Gamma$. Then by assumption
\[
\norm{\mathbbm{1} - \theta_n(\psi_n(g))}_2^2 = \tau_{k_n}(2\cdot\mathbbm{1} - \theta_n(\psi_n(g)) - \theta_n(\psi_n(g))^*) \xrightarrow{n\to\infty} 2.
\]
In other words, $(\theta_n\circ \psi_n)_n$ is a separating partial homomorphism to $\cU$, and hence $\Gamma$ is LEF by Lemma~\ref{lem:marked:LEF}.
\end{proof}

Similar to \cite[Lemma~1]{HS1}, it suffices to check extremal characters for the local character criterion to be satisfied, as we prove next.

\begin{lemma}
\label{lem:extremalcharacters}

    Let $\Gamma$ be a countable discrete group such that all extremal characters of $\Gamma$ satisfy the conclusion of the local character criterion. Then $\Gamma$ satisfies the local character criterion.
\end{lemma}
\begin{proof}
Fix a generating set $S$ of $\Gamma$, and let $\pi:\bF_S\to\Gamma$ be the corresponding marking. Let $\chi:\Gamma\to\C$ be a character. Fix $\eps>0$ and $g_1,\ldots, g_k\in\Gamma$. Choose rational numbers $q_1, \ldots, q_\ell$ with common denominator $m$ such that $\sum_{i=1}^\ell q_i = 1$, together with extremal characters $\chi_1, \ldots, \chi_\ell$ such that for all $1\leq j\leq k$:
\begin{equation}\label{eq:1}
\abs{\chi(g_j) - \sum_{i=1}^\ell q_i\chi_i(g_j)} < \eps.
\end{equation}
For each $1\leq i\leq \ell$, we can by assumption find an $S$-marked group $\Lambda_{i}$, together with a map $\psi_i:\Gamma\to\Lambda_i$ and a representation $\theta_i:\Lambda_i\to \U(k_i)$ such that for $1\leq i\leq \ell$ and $1\leq j,j'\leq k$,
\begin{equation}\label{eq:4}
\psi_i(g_jg_{j'}) = \psi_i(g_j)\psi_i(g_{j'}),
\end{equation}
and
\begin{equation}\label{eq:2}
\abs{\chi_i(g_j) - \tau_{k_i}(\theta_i(\psi_i(g_j)))} < \eps.
\end{equation}
Let $\Lambda\coloneqq \oplus_{i=1}^\ell \Lambda_i$, and define $\psi:\Gamma\to\Lambda: \psi(g) = (\psi_1(g), \ldots, \psi_\ell(g))$. Choose $N\in\N$ such that $\frac{mq_i}{k_i}N\in \Z$ for all $1\leq i\leq \ell$, and define
\[
\theta:\Lambda\to \U(mN): \theta(g^{(1)},\ldots, g^{(\ell)}) = \bigoplus_{i=1}^\ell \theta_i(g^{(i)})^{\oplus \frac{mq_i}{k_i}N}.
\]
Then
\begin{equation}\label{eq:3}
\tau_{mN}(\theta(\psi(g))) = \frac{1}{mN} \sum_{i=1}^\ell \frac{mq_i}{k_i}N \Tr_{k_i}(\theta_i(\psi_i(g))) = \sum_{i=1}^\ell q_i \tau_{k_i}(\theta_i(\psi_i(g))),
\end{equation}
where $\Tr_{k_i}$ denotes the non-normalized trace on $\bM_{k_i}$. Combining \eqref{eq:1}, \eqref{eq:2}, and \eqref{eq:3} then yields for all $1\leq j\leq k$:
\[
\abs{\chi(g_j) - \tau_{mN}(\theta(\psi(g_j)))} < 2\eps,
\]
which together with \eqref{eq:4} implies that $\chi$ satisfies the conclusion of the local character criterion.
\end{proof}

Finally, we note that if $\Gamma$ satisfies the local character criterion, then subgroups satisfy it ``relative to $\Gamma$'' in the following sense.

\begin{remark}\label{lem:inftofin}
    If $\Gamma$ satisfies the local character criterion~\ref{criterion}, then for every subgroup $\Sigma\leq\Gamma$ and every character $\chi:\Gamma\to\C$, the restricted character $\chi|_{\Sigma}$ satisfies the conclusion of the local character criterion for $\Sigma$. Indeed, the restriction of partial homomorphism from Lemma~\ref{rem:equivalent:criterion} to $\Sigma$ is a partial homomorphism with traces converging pointwise to $\chi|_{\Sigma}$. 
    Note however that this does not imply that $\Sigma$ satisfies the character criterion, since in general not all characters of $\Sigma$ can be acquired as restrictions of characters of $\Gamma$.  
\end{remark}

\subsection{A characterization of local HS-stability for amenable groups.}\label{sec:char:amenable}

We can now prove a local version of the character criterion for amenable groups. We first note the following:

\begin{remark}\label{rem:ultraproducts}
    When dealing with ultraproducts of matrix algebras (and more generally tracial von Neumann algebras), one usually only considers sequences in the direct product which are bounded in operator norm. In other words, the ultraproduct $\prod_\omega^{met} \bM_{k_n}$ is defined as the quotient of $\{(x_n)_n\in\prod \bM_{k_n}\mid \sup_n \norm{x_n}<\infty\}$ by the closed ideal $\{(x_n)_n\mid \lim_{n\to\omega} \norm{x_n}_2 = 0\}$. It was already observed by Sakai \cite{sakai:ultraproducts} that in this way, an ultraproduct of matrix algebras is always a tracial von Neumann algebra. Moreover, it is a well-known fact and easy exercise that the unitary elements of $\prod_\omega^{met} \bM_{k_n}$ are exactly $\prod_\omega^{met} \U(k_n)$, i.e. every $u\in \U(\prod_\omega^{met} \bM_{k_n})$ can be written as $u=(u_n)_n$ for some $u_n\in\U(k_n)$. We will always assume that an ultraproduct of matrix algebras is defined in the above way.
\end{remark}

Before proving the main result, Theorem~\ref{thm:char}, of this subsection below, we note the following, which is essentially a reformulation of \cite[Theorem~2.1]{HadwinAppr}.

\begin{proposition}[{\cite[Theorem 4.3]{actiontraces}}]
\label{thm:hadwin}
    Suppose $\Gamma$ is amenable and $\vphi,\psi: \Gamma\to\prod_\omega^{met} \U(k_n)$ are two homomorphisms such that $\tau_\omega\circ\vphi = \tau_\omega\circ \psi$. Then there exists a unitary $u\in \prod_\omega^{met} \U(k_n)$ such that $\vphi = u\psi u^*$.
\end{proposition}



We can now prove Theorem~\ref{intro:thm:char} from the introduction, characterizing local HS-stability for amenable groups.

\begin{theorem}[Theorem \ref{intro:thm:char}]
\label{thm:char}
Let $\Gamma$ be an amenable group generated by a countable set $S$, and let $\pi:\bF_S\to \Gamma$ be the corresponding marking. Then the following are equivalent:
\begin{enumerate}
    \item $\Gamma$ is locally HS-stable.
    \item $\Gamma$ satisfies the local character criterion~\ref{criterion}.
\end{enumerate}
\end{theorem}

\begin{proof}
$(1)\Rightarrow (2)$. Let $\chi:\Gamma\to\C$ be a character. Since $\Gamma$ is amenable, $\chi$ is embeddable by Connes' Theorem \cite{connes:hyperfinite}, i.e. there exists a non-principal ultrafilter $\omega$ on $\N$ together with a homomorphism $\vphi:\Gamma\to \prod_\omega^{met} \U(n)$ such that $\chi = \tau_\omega\circ\vphi$. By local HS-stability, $\vphi$ lifts to a homomorphism $\vphi_{alg}=(\vphi_n)_n:\Gamma\to \prod_\omega^{alg} \U(n)$. By Lemma \ref{lem:approximate:ultraproduct} and the definition of $\tau_\omega$, there exists a subsequence $(n_m)_m \in \omega$ such that $\psi_m \coloneqq \vphi_{n_m} : \Gamma \to \U(n_m)$ is a partial homomorphism and $\chi$ is a pointwise limit of $\tau_{n_m}\circ\psi_m$. Hence by Lemma~\ref{rem:equivalent:criterion}, $\Gamma$ satisfies the local character criterion.

$(2)\Rightarrow (1)$. Let $\omega$ be a non-principal ultrafilter on $\N$ and consider the homomorphism $\vphi = (\vphi_n)_n:\Gamma\to\prod_\omega^{met} \U(k_n)$. Denote by $\chi\coloneqq \tau_\omega\circ\vphi:\Gamma\to\C$ the associated character. By Lemma~\ref{rem:equivalent:criterion}, there exists a partial homomorphism $\psi_n:\Gamma\to \U(l_n)$ such that $\tau_{k_n}\circ \psi_n$ converges pointwise to $\chi$. By Lemma~\ref{lem:dims}, we can assume that $l_n=k_n$, hence we can view $\vphi$ and $\psi = (\psi_n)_n$ as homomorphisms from $\Gamma$ to $\prod_\omega^{met} \U(k_n)$. By construction, we then have $\tau_\omega \circ \vphi = \tau_\omega \circ \psi$. Theorem~\ref{thm:hadwin} thus implies that there exists a unitary $u\in\prod_\omega^{met} \U(k_n)$ such that $\vphi = u\psi u^*$. 
The latter is by construction a homomorphism to $\prod_\omega^{alg}\U_{k_n}$ which lifts $\vphi$, and we conclude that $\Gamma$ is locally HS-stable by Lemma~\ref{lem:stable:ultraproduct}(2).
\end{proof}

\begin{remark}
    This proof follows the same general idea as the original proof of Hadwin and Shulman showing that an amenable group $\Gamma$ is HS-stable if and only every character/trace on $\Gamma$ is a pointwise limit of normalized traces of finite-dimensional unitary representations \cite[Theorem~4]{HS2}. Furthermore, by observing that upon replacing partial homomorphisms by genuine homomorphisms, we would get finite-dimensional representations of $\Gamma$, and the above proof recovers the original aforementioned result for HS-stability.
\end{remark}

We note the following easy corollary to Theorem~\ref{thm:char}, which can be viewed as the local analogue of \cite[Proposition~6.2]{ES:HS}.

\begin{proposition}[Theorem \ref{intro:thm:amenable}(1)]
\label{prop:flexible}

    Assume $\Gamma$ is amenable. Then $\Gamma$ is flexibly locally HS-stable if and only if $\Gamma$ is locally HS-stable.
\end{proposition}

\begin{proof}
Assume $\Gamma$ is flexibly locally HS-stable, and let $\chi:\Gamma\to\C$ be a character. Similar to the proof of $(1)\Rightarrow (2)$ in Theorem~\ref{thm:char}, there exists an approximate homomorphism $\vphi_n\colon\Gamma\to \U(k_n)$ such that $\chi$ is a pointwise limit of $\tau_{k_n}\circ\vphi_n$. By flexible local HS-stability, for some $m_n\ge k_n$ such that $\frac{m_n}{k_n}\to 1$, there exists a partial homomorphism $\psi_n\colon\Gamma\to \U(m_n)$ such that for all $g\in\Gamma$
$$||\vphi_n(g)-P_n\psi_n(g)P_n||_2\to 0,$$
where $P_n$ denotes the projection onto the first $k_n$ coordinates. Direct computations (see e.g. \cite[Proposition~6.2]{ES:HS}) show that $\tau_{m_n}\circ\psi_n$ converges pointwise to $\chi$. Therefore by Lemma~\ref{rem:equivalent:criterion}, $\Gamma$ satisfies the local character criterion, and thus by Theorem~\ref{thm:char}, $\Gamma$ is locally HS-stable.
\end{proof}

\subsection{Products with amenable groups}

Next, we observe that \cite[Proposition~C]{IS-exotic} implies immediately that local HS-stability is preserved under taking a direct product with an \emph{amenable} locally HS-stable group (cf. \cite[Corollary~D]{IS-exotic}).

\begin{proposition}[Theorem \ref{intro:thm:amenable}(2)]
\label{prop:product:amenable}

    Assume $\Gamma$ and $\Sigma$ are ((very) flexibly) locally HS-stable, and assume $\Gamma$ is amenable. Then $\Gamma \times \Sigma$ is ((very) flexibly) locally HS-stable.
\end{proposition}

\begin{proof}
Assume $\vphi_n: \Gamma \times \Sigma \to \U(k_n)$ is an approximate homomorphism, and consider the corresponding homomorphism $\vphi: \Gamma \times \Sigma \to \prod_\omega^{met} \U(k_n)$ for a fixed non-principal ultrafilter $\omega$ on $\N$. Then $P \coloneqq \vphi(\Gamma)''$ and $Q\coloneqq \vphi(\Sigma)''$ are commuting separable von Neumann subalgebras of $\prod_\omega^{met} \bM_{k_n}$ (cf.~Remark~\ref{rem:ultraproducts}), and $P$ is amenable. By \cite[Proposition~C]{IS-exotic}, there exist commuting von Neumann subalgebras $P_n,Q_n\sub \bM_{k_n}$ such that $P\subset \prod_\omega^{met} P_n$ and $Q\subset \prod_\omega^{met} Q_n$. Hence, applying ((very) flexible) local HS-stability for $\Gamma$ and $\Sigma$, with ranges in $\prod_\omega^{met} P_n$ and $\prod_\omega^{met} Q_n$ respectively, finishes the proof.
\end{proof}

\begin{example}
    Proposition~\ref{prop:product:amenable} immediately gives some easy examples of non-amenable groups which are locally HS-stable but not HS-stable. Indeed, the direct product of $\bF_2$ with any amenable locally HS-stable group $\Gamma$ is locally HS-stable, but it will not be HS-stable if $\Gamma$ is not HS-stable. For instance, $S_\infty\times\bF_2$ produces such an example. Using Theorem \ref{intro:thm:full}, we can also take $\Gamma = [[T]]'$ as in the statement, which yields finitely generated examples.
\end{example}

\subsection{Very flexible local HS-stability}

In this subsection, we show that for amenable groups, very flexible local HS-stability is equivalent to being LEF, which is an analogue of the corresponding statement for usual stability from \cite{ES:HS} (recall from Lemma \ref{lem:LEF:hyperlinear} that the local analogue of MAP is equivalent to LEF). One tool we will use is the so-called diagonal product of markings, which we introduce first.
\begin{definition}
    Fix a sequence of markings $\pi_n\colon\mathbb{F}_S\to\Lambda_n$. The image of the canonical map $\pi\colon\mathbb{F}_S\to\prod_{n}\Lambda_n$ given by $\pi(s)=(\pi_n(s))_n$ for $s\in S$, is called the \emph{diagonal product} of $\pi_n$. We will denote the diagonal product by $\ot\Lambda_n$ and the corresponding marking by $\ot\pi_n$. 
\end{definition}

Note that $\ker  \ot\pi_n=\bigcap \ker\pi_n $.
For future reference, we record the following lemma which follows easily from the definitions.

\begin{lemma}\label{lem:diagmarked}
    Suppose $\pi_n\colon \mathbb{F}_S\to F_n$ is a sequence of markings converging to $\pi\colon\mathbb{F}_S\to \Gamma$.
    \begin{enumerate}\setcounter{enumi}{1}
        \item There is a quotient of marked groups $q:\ot_{n\in\mathbb{N}}F_n\to \Gamma$ whose kernel is given by $\left(\ot_{n\in\mathbb{N}}F_n\right)\cap\left( \bigoplus_{n\in\mathbb{N}} F_n\right)$.
        \item The groups $\Lambda_\ell = \ot_{n\geq \ell}F_n$ converge to $\Gamma$ in the space of $S$-marked groups.
    \end{enumerate}
\end{lemma}

\begin{proof}
(1) This is \cite[Lemma~4.6]{KP}.

(2) A word $w\in\bF_S$ represents the trivial element in $\Lambda_\ell$ if and only if $w$ is trivial in $F_m$ for all $m\geq \ell$. Hence the conclusion follows immediately from \eqref{eq:markedconvergence}.
\end{proof}

The following result is an analogue in the local setting of \cite[Corollary~6.6]{ES:HS}. In fact, our proof will start similarly, until the point where residual finiteness of the group is used in \cite[Corollary~6.6]{ES:HS}. At this point, we will involve the diagonal products introduced above to leverage LEF and finish the proof in the space of marked groups.

\begin{theorem}[Theorem \ref{intro:thm:amenable}(3)]
\label{thm:veryflex}

Let $\Gamma$ be an amenable group. Then $\Gamma$ is very flexibly locally HS-stable if and only if $\Gamma$ is LEF.
\end{theorem}

\begin{proof}
If $\Gamma$ is very flexibly locally HS-stable, it follows from Lemma~\ref{lem:stable:LEF} that $\Gamma$ is LEF.

Conversely, assume that $\Gamma$ is LEF, and suppose $\vphi_n:\Gamma\to \U(k_n)$ is an approximate homomorphism. Let $\omega$ be a non-principal ultrafilter on $\N$ and consider the homomorphism $\vphi:\Gamma\to \prod_\omega^{met} \U(k_n)$. By definition of the universal $C^*$-algebra, $\vphi$ extends uniquely to a $^*$-homomorphism $\vphi: C^*(\Gamma)\to \prod_\omega^{met} \bM_{k_n}$. Since $C^*(\Gamma)$ is amenable, it has the lifting property by \cite[Theorem~3.10]{CE:LP}, hence $\vphi$ lifts to a sequence of ucp maps $\theta = (\theta_n)_n: C^*(\Gamma)\to \prod_{n \in \N} \bM_{k_n}$. By Stinespring Dilation (see, e.g., \cite[Theorem~3.6]{takesaki} for this and further background on ucp maps), there exist $^*$-homomorphisms $\pi_n:C^*(\Gamma)\to B(H)$ for some infinite-dimensional Hilbert space $H$ with a fixed orthonormal basis $(e_j)_{j\in\N}$ such that 
\begin{equation}\label{eq:Stinespring}
\theta_n = P_{k_n}\pi_nP_{k_n}
\end{equation}
where $P_{k_n}$ is the projection of $H$ onto the $k_n$-dimensional subspace generated by $e_1, \ldots, e_{k_n}$, and where we identify $\bM_{k_n} = B(P_{k_n}H)$.

Fix a generating set $S$ for $\Gamma$. Since $\Gamma$ is LEF, we can find a sequence of finite $S$-marked groups $F_m$ converging to $\Gamma$ in the space of $S$-marked groups. Consider the diagonal products $\Lambda_\ell = \ot_{m\geq \ell}F_m$.
Lemma~\ref{lem:diagmarked}(1) gives us quotient maps $q_\ell:\Lambda_\ell\to \Gamma$, and hence composing with $\pi_n$ we get homomorphisms $\pi_n\circ q_\ell:\Lambda_\ell\to \U(H)$. We note that each $\Lambda_\ell$ is residually finite being a subgroup of a product of finite groups, and amenable by Lemma~\ref{lem:diagmarked}(1). Hence $C^*(\Lambda_\ell)$ is RFD by \cite{BL:RFD}, and thus by \cite{Had:Lift} there exist finite rank projections $Q_i^{n,\ell}\geq P_{k_n}$ and $^*$-homomorphisms
\begin{equation}\label{eq:findimconv}
\rho_i^{n,\ell}:C^*(\Lambda_\ell)\to B(Q_i^{n,\ell}H)
\end{equation}
such that for every $n,\ell\in\N$ and every $a\in C^*(\Lambda_\ell)$, $\rho_i^{n,\ell}(a)$ converges to $\pi_n(q_\ell(a))$ in the strong operator topology as $i\to\infty$. 

Next, by Lemma~\ref{lem:diagmarked}(2), $\Lambda_\ell$ converges to $\Gamma$ in the space of $S$-marked groups. Write $\sigma:\bF_S\to\Gamma$ and $\sigma_\ell:\bF_S\to\Lambda_\ell$ for the respective markings. 
We can apply the second part of Remark~\ref{rem:convergence:quotient} to the identity maps $\psi_\ell\colon\Lambda_\ell\to\Lambda_\ell$ and get a partial homomoprhism $f_\ell\colon \Gamma\to \Lambda_\ell$ such that for any $w\in \bF_S$ we have $f_\ell(\sigma(w))=\sigma_\ell(w)$ for $\ell$ large enough. 
We know that for each $g\in \Gamma$ and $n$ large enough $\rho^{n,n}_i(f_n(g))$ converges in the strong operator topology to $\pi_n(g)$ when $i\to\infty$. Hence there exist $i_n$ such that for any $g\in\Gamma$ 
$$\sum_{j=1}^{k_n}\norm{[\pi_n(g) - \rho^{n,n}_{i_n}(f_n(g))](e_j)}^2\to 0.$$
By construction, $\psi_n=\rho_{i_{n}}^{n,n}\circ f_n: \Gamma\to \U(Q_{i_{n}}^{n,n} H)$ is a partial homomorphism and for any $g\in\Gamma$ we have
\begin{align*}
\Dn(\vphi_n(g),\psi_n(g)) &= \norm{\vphi_n(g) - P_{k_n}\psi_n(g)P_{k_n}}_2 \\
&\leq \norm{\vphi_n(g) - \theta_n(g)}_2 + \norm{P_{k_n}\pi_n(g)P_{k_n} - P_{k_n}\psi_n(g)P_{k_n}}_2\\
&= \norm{\vphi_n(g) - \theta_n(g)}_2 + \left(\frac{\sum_{j=1}^{k_n}\norm{[P_{k_n}\pi_n(g)P_{k_n} - P_{k_n}\psi_n(g)P_{k_n}](e_j)}^2}{k_n}\right)^{1/2}\\
&\xrightarrow[]{n\to\infty} 0.
\end{align*}
We conclude that $\Gamma$ is very flexibly locally HS-stable, finishing the proof.
\end{proof}

\subsection{Some remarks on weak local HS-stability}

In the literature on P-stability and HS-stability there is a further notion that is often studied, which is potentially useful when considering the connection between stability and the corresponding approximation properties. Namely, when considering the (local) stability question, we only consider approximate homomorphisms which separate points ``maximally'' (cf. Remark \ref{rmk:amplification}):

\begin{definition}
\label{def:weak}

An approximate homomorphism $\vphi : \Gamma \to \cS$ is called a \emph{sofic approximation} if
$$\lim\limits_{n \to \infty} \dHamm(\vphi_n(g), \id_{k_n}) = 1.$$
for all $e \neq g \in \Gamma$. We say that $\Gamma$ is \emph{weakly P-stable} if every sofic approximation is close to a sequence of homomorphisms.

An approximate homomorphism $\vphi : \Gamma \to \cU$ is called a \emph{hyperlinear approximation} if
$$\lim\limits_{n \to \infty} \| \vphi_n(g) - \mathbbm{1} \|_2 = \sqrt{2}.$$
for all $e \neq g \in \Gamma$. We say that $\Gamma$ is \emph{weakly HS-stable} if every hyperlinear approximation is close to a sequence of homomorphisms.
\end{definition}

We can similarly formulate local versions of weak stability (cf. \cite[Definition 1.3]{bradford}):

\begin{definition}
\label{def:weak:local}

A group is \emph{weakly locally P-stable} if every sofic approximation $\vphi_n : \Gamma \to \Sym(k_n)$ is close to a partial homomorphism $\psi_n : \Gamma \to \Sym(k_n)$. It is \emph{weakly locally HS-stable} if every hyperlinear approximation $\vphi_n : \Gamma \to \U(k_n)$ is close to a partial homomorphism $\psi_n : \Gamma \to \U(k_n)$.
\end{definition}

The reason why these properties can be useful in the the connection between stability and approximation properties is the following fact.

\begin{lemma}
\label{lem:weaklystable:LEF}

If $\Gamma$ is hyperlinear and weakly HS-stable, then $\Gamma$ is MAP. If $\Gamma$ is hyperlinear and weakly locally HS-stable, then $\Gamma$ is LEF.
\end{lemma}

\begin{proof}
This follows from the same argument as Lemma \ref{lem:stable:LEF}: indeed if a group is hyperlinear, it admits not only a separating approximate representation, but even a hyperlinear approximation (cf. Remark \ref{rmk:amplification}).
\end{proof}

For finitely generated amenable groups, both of these notions are well-understood in the case of P-stability:

\begin{proposition}[{\cite[Theorem 7.2]{AP:commuting}, \cite[Lemma 2.7]{bradford}}]

A finitely generated amenable group is weakly P-stable if and only if it is residually finite. It is weakly locally P-stable if and only if it is LEF.
\end{proposition}

We now consider the analogous statements for (local) HS-stability, without the assumption of finite generation.

\begin{proposition}[Theorem \ref{intro:thm:amenable}(4)]
\label{prop:weaklylocallyHSstable}

An amenable group is weakly HS-stable if and only if it is MAP. It is weakly locally HS-stable if and only if it is LEF.
\end{proposition}

The key ingredient is the following uniqueness result for hyperlinear approximations of a given amenable group $\Gamma$. The similar result for sofic approximations was established by Elek and Szab{\'o} \cite{ES:sofic}, see also \cite[Theorem 3.12]{actiontraces} for a more general fact that removes the finite generation hypothesis in \cite{ES:sofic}. For hyperlinear approximations, the result easily follows from \cite{HadwinAppr}, as reformulated in Proposition~\ref{thm:hadwin}.

\begin{corollary}
\label{cor:hyperlinear:conjugate}

Let $\vphi_n, \psi_n : \Gamma \to \U(k_n)$ be hyperlinear approximations of an amenable group. Then there exists a sequence of unitaries $u_n \in \U(k_n)$ such that $(\vphi_n)_n$ is close to $(u_n \psi_n u_n^*)_n$.
\end{corollary}
\begin{proof}
    By definition, given a hyperlinear approximation $\vphi:\Gamma\to \prod_\omega \U(k_n)$, we have $\tau_\omega\circ \vphi = \delta_e$, hence the result follows immediately from Proposition~\ref{thm:hadwin}.
\end{proof}

Finally, we need the following fact:

\begin{lemma}
\label{lem:dims2}

Let $(k_n)_n \subset \N$ be an increasing sequence. If $\Gamma$ is MAP, then there exists a sequence of homomorphisms $\vphi_n : \Gamma \to \U(k_n)$ that form a hyperlinear approximation. If $\Gamma$ is LEF, then there exists a partial homomorphism $\vphi_n : \Gamma \to \U(k_n)$ that forms a hyperlinear approximation.
\end{lemma}

\begin{proof}
If $\Gamma$ is MAP, then it admits a sequence of homomorphisms $\Gamma \to \U(k_n)$ that are injective on an exhaustion of $\Gamma$ by finite sets, and if $\Gamma$ is LEF, then it admits a local embedding $\Gamma \to \cU$, by Lemma \ref{lem:LEF:hyperlinear}.

Both statements then follow from an amplification argument analogous to Lemma \ref{lem:dims}, noticing that direct sums of (partial) homomorphisms are (partial) homomorphisms. See \cite{ES:amplification} for a slightly weaker statement, and also \cite[Proposition~6.1]{AP:commuting} and \cite[Lemma~2.12]{bradford} for the full statement in the case of permutations.
\end{proof}

Proposition~\ref{prop:weaklylocallyHSstable} now easily follows from a combination of the aforementioned results:

\begin{proof}[Proof of Proposition~\ref{prop:weaklylocallyHSstable}]

If $\Gamma$ is amenable, hence hyperlinear, and weakly HS-stable, respectively weakly locally HS-stable, then it is MAP, respectively LEF, by Lemma \ref{lem:weaklystable:LEF}.

Conversely, suppose that $\Gamma$ is MAP, and let $\vphi_n : \Gamma \to \U(k_n)$ be a hyperlinear approximation. By Lemma \ref{lem:dims2} there exists a sequence of homomorphisms $\psi_n : \Gamma \to \U(k_n)$ which is also a hyperlinear approximation. By Corollary~\ref{cor:hyperlinear:conjugate}, there exist unitaries $(u_n)_n$ such that $(\vphi_n)_n$ is close to $(u_n \psi_n u_n^*)_n$. The latter is a sequence of homomorphisms and thus $\Gamma$ is weakly HS-stable.

Similarly if $\Gamma$ is LEF we use Lemma \ref{lem:dims2} to obtain a partial homomorphism with prescribed range that is also a hyperlinear approximation, and conclude again by Corollary~\ref{cor:hyperlinear:conjugate}.
\end{proof}

\begin{remark}
One can similarly define flexible versions of weak (local) HS-stability, and it is easy to see that Proposition~\ref{prop:weaklylocallyHSstable} holds for those as well. Indeed, Lemmas~\ref{lem:weaklystable:LEF} and \ref{lem:stable:LEF} generalize easily to proving that a hyperlinear weakly very flexibly stable group is MAP. Thus, for an amenable group, we have the chain of implications:
\begin{align*}
    \text{weakly HS-stable} &\Rightarrow \text{weakly flexibly HS-stable} \Rightarrow \text{weakly very flexibly HS-stable} \\
    &\xRightarrow{\text{Lemma \ref{lem:weaklystable:LEF}}} \text{MAP} \xRightarrow{\text{Proposition \ref{prop:weaklylocallyHSstable}}} \text{weakly HS-stable}.
\end{align*}
The case for the local version is completely analogous, as is the case for permutations.
\end{remark}

\section{Relation with (local) P-stability}\label{sec:examples}

In this section we observe some connections between (local) P-stability and (local) HS-stability, and we apply Theorem~\ref{thm:char} to produce additional examples of amenable locally HS-stable groups.

\subsection{From P-stability to HS-stability}

Similar to the (local) character criterion, there is a criterion for (local) permutation stability in terms of \emph{invariant random subgroups}. Specifically, an amenable group is P-stable if and only if every IRS is cosofic \cite[Theorem~7.10]{BLT:IRS}, and it is locally P-stable if and only if every IRS is partially cosofic \cite[Theorem~4.8]{bradford}. We briefly review the necessary definitions for our purposes and refer to \cite{BLT:IRS} and \cite{bradford} for further definitions and results.

We fix a countable discrete group $\Gamma$ and denote by $\Sub(\Gamma)$ the space of all subgroups of $\Gamma$. With the induced topology from $\{0,1\}^\Gamma$, $\Sub(\Gamma)$ becomes a compact metrizable space (recall that we assume that $\Gamma$ is countable). Moreover, it comes with a canonical continuous action of $\Gamma$ by conjugation. An \emph{invariant random subgroup (IRS)} of $\Gamma$ is a conjugation-invariant probability measure on $\Sub(\Gamma)$. Endowed with the weak$^*$-topology, a sequence of IRS's $\mu_n$ converges to an IRS $\mu$ if and only if $\int f\,d\mu_n\to \int f\,d\mu$ for every continuous function $f:\Sub(\Gamma)\to \R$. 

Given a probability measure preserving (p.m.p.) action $\alpha$ of $\Gamma$ on a probability space $(X,\nu)$, we get an IRS and a character in the following way. 
\begin{enumerate}
    \item Consider the stabilizer map $X\to \Sub(\Gamma):x\mapsto \Stab_\Gamma(x)$. Then the pushforward of $\nu$ through the stabilizer map is an IRS $\mu_\alpha$ for $\Gamma$ given by $\mu_\alpha(A) = \nu(\Stab^{-1}(A))$ for $A\subset\Sub(\Gamma)$. In fact, every IRS of $\Gamma$ arises in this way \cite[Proposition 1.4]{irs:action}.
    \item Define $\chi_\alpha:\Gamma\to \C$ by $\chi(g) = \nu(\Fix(g))$. Then $\chi$ is a character, which we also refer to as the \emph{fixed-point character of $\alpha$}.
\end{enumerate}

We note the following lemma, which is surely known to experts, yet seems not to have appeared in the literature so far. 

\begin{lemma}\label{lem:IRStoChar}
    Given a p.m.p. action $\alpha$ of $\Gamma$ on a probability space $(X,\nu)$, consider its associated IRS $\mu_\alpha$ and fixed-point character $\chi_\alpha$ as above. If $\mu_\alpha$ is a weak$^*$-limit of atomic IRS's supported on finite-index subgroups, then $\chi_\alpha$ is a pointwise limit of traces of finite-dimensional representations.
\end{lemma}
\begin{proof}
For notational convenience we suppress the subscript $\alpha$ from the notations. Let $\mu_n$ be a sequence of atomic IRS's supported on finite-index subgroups converging to $\mu$ in the weak$^*$-topology. By \cite[Section~4]{BLT:IRS}, we can assume that each $\mu_n$ is moreover finitely supported, i.e., $\mu_n$ arises from a p.m.p. action of $\Gamma$ on a finite probability space $(X_n,\nu_n)$. Denote by $\chi_n$ the associated fixed-point character for this action. For $g\in\Gamma$, we denote by $C_g = \{H\in\Sub(\Gamma)\mid g\in H\}$ the set of subgroups of $\Gamma$ containing $g$. One can then calculate for $g\in \Gamma$:
\begin{align*}
\chi(g) &= \nu(\{x\in X\mid gx = x\})\\
&= \nu(\{x\in X\mid g\in \Stab(x)\}\\
&= \nu(\Stab^{-1}(C_g))\\
&= \mu(C_g)\\
&= \lim_{n\to\infty} \mu_n(C_g)\\
&= \lim_{n\to\infty} \chi_n(g)
\end{align*}
where the penultimate equality follows from the fact that $C_g\subset\Sub(\Gamma)$ is clopen, meaning its indicator function is continuous. The proof will thus be complete if we can show that each $\chi_n$ is a pointwise limit of traces of finite-dimensional representations. 

Fix $n\in\N$, $\eps>0$, and enumerate $X_n = \{x_1, \ldots, x_{k_n}\}$. Choose a $\Gamma$-invariant function $q\colon X_n\to \mathbb{Q}$ with positive rational values such that $\sum_{i=1}^{k_n} q(x_i) = 1$ and $\sum_{i=1}^{k_n} \abs{\nu_n(x_i) - q(x_i)}<\eps$. Write $q(x_i)=\frac{t_i}{m}$ for $t_i, m\in\mathbb{Z}$. Consider $\oplus_{i=1}^{k_n} \bM_{t_i}\subseteq \bM_m$ and the $m$-dimensional representation $\rho$ of $\Gamma$ which permutes the blocks according to its action on $X_n$. Then it is easy to check that $\tau_m(\rho(g)) = \sum_{x\in\Fix(g)} q(x)$, hence $\abs{\chi_n(g) - \tau_m(\rho(g))}<\eps$. This finishes the proof of the lemma.
\end{proof}

\begin{remark}
The converse of the previous lemma is false. Indeed, by \cite[Theorem 5.23]{ES:HS}, there exists a finitely generated amenable group $\Gamma$ that is HS-stable but not P-stable. Since it is not P-stable, by the IRS criterion \cite{BLT:IRS} there exists an IRS $\mu$ that is not the weak$^*$-limit of atomic IRS's supported on finite-index subgroups of $\Gamma$. This IRS comes from a p.m.p. action $\alpha$ \cite[Proposition 1.4]{irs:action} with an associated fixed-point character $\chi$. Since $\Gamma$ is HS-stable, $\chi$ is a pointwise limit of traces of finite-dimensional representations by the character criterion \cite[Theorem 4]{HS2}.
\end{remark}

\begin{corollary}[Proposition~\ref{intro:prop:PtoHS}]
\label{cor:PtoHS}

Suppose $\Gamma$ is an amenable group such that every extremal character can be realized as the fixed-point character of a p.m.p. action. If $\Gamma$ is (locally) P-stable, then $\Gamma$ is (locally) HS-stable.
\end{corollary}

\begin{proof}
The statement for ordinary stability follows directly from Lemma \ref{lem:IRStoChar}, by considering the IRS criterion \cite{BLT:IRS} and the character criterion \cite[Theorem 4]{HS2} as they apply to $\Gamma$. The statement for local stability follows from the local IRS criterion \cite{bradford} and the local character criterion (Theorem \ref{thm:char}), by applying Lemma \ref{lem:IRStoChar} to the appropriate sequence of groups converging to $\Gamma$ in the space of marked groups.
\end{proof}

In the next subsection, we will apply this corollary to deduce local HS-stability for certain topological full groups whose local P-stability was established in \cite{bradford}. However, we hope that the same principle may help to provide further examples of (locally) HS-stable groups. A good candidate seems to be a class of branch groups containing Grigorchuk's group \cite{grigorchuk:group}, whose P-stability was proved in \cite{zheng}. We also refer to \cite{grigorchuk:characters} for a study of a large class of characters for Grigorchuk's group, all of which arise as fixed-point characters of p.m.p. actions.

\subsection{Topological full groups of minimal Cantor subshifts}

One class of examples to which the local character criterion~\ref{criterion} applies are certain topological full groups of minimal Cantor systems, which were introduced in \cite{fullgroups:def}. This follows from a combination of an upcoming announced result \cite{DM} and results in the literature, so we will not go into full detail here, and indicate the necessary references below. 

Given a homeomorphism $T$ of the Cantor set $X$, the topological full group consists of all homeomorphisms $g$ of $X$ such that there exists a \emph{continuous} function $f_g:X\to\Z$ such that $g(x) = T^{f_g(x)}(x)$ for all $x\in X$. We call $T$ minimal if all its orbits are dense. We refer to \cite{GM:TopFullGp,bradford} for further references, and for some relevant properties and constructions related to these topological full groups. If $(X,T)$ is isomorphic to a subshift (see \cite{LM:shifts}), we also call $(X,T)$ a minimal Cantor subshift. We point out the following results.

\begin{theorem}\label{thm:fullgrpprops}
    Let $(X,T)$ be a minimal Cantor system.
    \begin{enumerate}
        \item (\cite{fullgroups:amenable}) $[[T]]$ is amenable.
        \item (\cite{GM:TopFullGp}) $[[T]]$ is LEF.
        \item (\cite{fullgroups:basics}) $[[T]]'$ is an infinite simple group, and it is finitely generated if $(X,T)$ is a minimal Cantor subshift.
        \item (\cite{fullgroups:basics}) There are uncountably many non-isomorphic groups of the form $[[T]]'$ where $(X,T)$ is a minimal Cantor subshift.
        \item (\cite{bradford}) If $(X,T)$ is a minimal Cantor subshift, then $[[T]]'$ is locally P-stable.
    \end{enumerate}
\end{theorem}

Recently, Dudko and Medynets \cite{DM} have announced a proof that for a minimal Cantor system $(X,T)$, all extremal characters of $[[T]]'$ arise as fixed point characters of p.m.p. actions. Combining this with Corollary~\ref{cor:PtoHS} and Theorem~\ref{thm:fullgrpprops}, one thus gets the following result.

\begin{theorem}[Theorem~\ref{intro:thm:full}]\label{thm:full}
    Let $(X,T)$ be a minimal Cantor subshift. Then $[[T]]'$ is locally HS-stable. In particular, there are uncountably many non-isomorphic finitely generated groups which are locally HS-stable, but not HS-stable.
\end{theorem}

\begin{remark}
    We note that in \cite{fullgroups:characters} a classification of characters of certain full groups was already carried out. However, those full groups are not amenable, and therefore a character classification does not help towards determining whether they are (locally) HS-stable. Moreover, we point out that the full groups from \cite{fullgroups:characters} are ``character rigid'' in the sense that, besides the regular character, all characters come from the abelianization of the group. Hence they behave very differently from the amenable topological full groups considered above, which admit a much richer character theory.
\end{remark}

\section{Failure of local HS-stability for groups with Property (T)}\label{sec:T}

In contrast with amenability, infinite property (T) groups are never P-stable, respectively HS-stable, assuming they are sofic, respectively hyperlinear. This was proven by Becker and Lubotzky in \cite{BL:T}, and it is natural to wonder whether the same holds for local stability (note that property (T) groups need not be finitely presented \cite[Section 3.4]{T}). This question is attributed to Lubotzky \cite[Question~6.5]{bradford}, and the goal of this section is to answer it affirmatively:

\begin{theorem}[Theorem \ref{intro:thm:T}]
\label{thm:T}
Assume $\Gamma$ is an infinite group with property (T).
\begin{enumerate}
    \item If $\Gamma$ is hyperlinear, then it is not locally HS-stable.
    \item If $\Gamma$ is sofic, then it is not locally P-stable.
\end{enumerate}
\end{theorem}

Our proof will rely on Shalom's Theorem \cite{fpT}, which states that every property (T) group is a quotient of a finitely presented property (T) group. Therefore, we start by investigating LEF-quotients of finitely presented property (T) groups. For a residually finite property (T) group $\Gamma$, one can construct a sequence of irreducible representations with increasing dimensions by taking irreducible components of the canonical action of $\Gamma$ on $\ell^2(\Gamma/\Gamma_{k})\ominus \ell^2(\Gamma/\Gamma_{k-1})$, where $(\Gamma_k)_k$ is a strictly decreasing sequence of finite index subgroups in $\Gamma$. We adapt this construction to prove more generally that if a finitely presented property (T) group $\Gamma_0$ factors onto an infinite LEF group $\Gamma$, then $\Gamma_0$ admits irreducible finite-dimensional representations of arbitrarily high dimension which descend to a partial homomorphism of $\Gamma$.

\begin{lemma}\label{lem:irreps}
Assume $\Gamma$ and $\Gamma_0$ are infinite property (T) groups, $\Gamma$ is LEF, $\Gamma_0$ is finitely presented, and there is a surjective homomorphism $f:\Gamma_0\to\Gamma$. Then the following hold.
\begin{enumerate}[itemsep=5pt]
    \item There exist a sequence of positive integers $k_n\to\infty$ and sequences of maps $\rho_n:\Gamma\to \U(k_n)$ and $\sigma_n:\Gamma_0\to \U(k_n)$ such that:
        \begin{itemize}[itemsep=3pt,topsep=3pt]
            \item $(\rho_n)_n$ is a partial homomorphism of $\Gamma$,
            \item $\sigma_n$ is an irreducible representation of $\Gamma_0$ for every $n$, and
            \item for every $g\in\Gamma_0$, there exists $N\in\N$ such that for all $n\geq N$: $\rho_n(f(g)) = \sigma_n(g)$.
        \end{itemize}
    \item There exist a sequence of positive integers $t_n \to \infty$ and sequences of maps $\alpha_n:\Gamma\to \Sym(t_n)$ and $\beta_n:\Gamma_0\to \Sym(t_n)$ such that 
        \begin{itemize}[itemsep=3pt,topsep=3pt]
            \item $(\alpha_n)_n$ is a partial homomorphism,
            \item $\beta_n$ is a transitive action of $\Gamma_0$ for all $n$, and
            \item for every $g\in\Gamma_0$, there exists $N\in\N$ such that for all $n\geq N$: $\alpha_n(f(g)) = \beta_n(g)$.
        \end{itemize}
\end{enumerate}
\end{lemma}
\begin{proof}
Fix finite generating sets $S_0$ and $S=f(S_0)$ for $\Gamma_0$ and $\Gamma$ respectively, and fix a finite set of relations $R$ such that $\Gamma_0 = \langle S_0\mid R\rangle$. We also denote by $B(n)$, respectively $B_0(n)$, the ball of radius $n$ inside $\Gamma$ for the word metric induced by $S$, respectively inside $\Gamma_0$ for $S_0$. Note that by construction $f(B_0(n))=B(n)$.

Since $\Gamma$ is LEF, there exist finite groups $F_n$ and an injective  partial homomorphism (see Remark~\ref{rk:partial:localemb}) $\vphi_n:\Gamma\to F_n$. By passing to a subsequence if necessary, we can assume that $\vphi_n|_{B(n)}$ is injective and for all $g,h\in B(n)$: $\vphi_n(gh)=\vphi_n(g)\vphi_n(h)$.

Construct homomorphisms $\psi_n\colon \bF_{S_0}\to F_n$ by $\psi_n(s)=\vphi_n(f(s))$ for $s\in S_0$. For large enough $n$ we then have $R\subseteq B_{\bF_{S_0}}(n)$, i.e. $\psi_n(R)=\{1_{F_n}\}$, and hence $\psi_n$ can be considered as a homomorphism from $\Gamma_0$. Since any finitely generated group admits only finitely many quotients of a given order, it moreover follows that $|F_n| \to \infty$. Note that $\psi_n|_{B_0(n)} = \vphi_n\circ f|_{B_0(n)}$. 
If necessary, we can then adjust $\vphi_n$ outside of $B(n)$ to make sure that $\vphi_n(\Gamma)\subseteq\psi_n(\Gamma_0)$. This allows us to assume that $F_n=\psi_n(\Gamma_0)$.

Let $\tau_n$ be the regular permutation representation of $F_n$. We let $\alpha_n \coloneqq \tau_n \circ \vphi_n$ and $\beta_n \coloneqq \tau_n \circ \psi_n$. Then the second part of the lemma is satisfied. Note that for this part we did not need the assumption of property (T).

For the first part, we need some extra constructions to ensure that the resulting representations are irreducible. We denote the left regular representation of $F_n$ by $\lambda_n$, and note that $\lambda_n\circ\psi_n$ is a finite-dimensional representation of $\Gamma_0$ for every $n$. We denote the kernel of $\lambda_n\circ\psi_n$ by $K_n$, and the kernel of $f$ by $K_0$.

We construct $\rho_n$ and $\sigma_n$ by induction. Let $\sigma_1$ be any irreducible component of $\lambda_1\circ\psi_1$, and let $\rho_1$ be $\lambda_1\circ \vphi_1$ with its range restricted to this irreducible component. 

Now assume $\rho_1,\ldots,\rho_n$ and $\sigma_1,\ldots,\sigma_n$ have been constructed. Since $\Gamma$ is infinite and $\cap_{i=1}^n K_i$ is a finite index subgroup, there exists $g\in\left(\cap_{i=1}^n K_i\right)\setminus K_0$. By assumption $\psi_m(g)=\vphi_m(f(g))\neq\mathbbm{1}$ for $m$ large enough, and thus we can find $m_n\in\N$ such that $g\notin K_{m_n}$ and $\lambda_{m_n}\circ\psi_{m_n}(g)\neq\mathbbm{1}$. In particular, we can find an irreducible component $\sigma_{n+1}$ of $\lambda_{m_n}\circ\psi_{m_n}$ such that $\sigma_{n+1}(g)\neq\mathbbm{1}$. Since by construction $\sigma_i(g)=\mathbbm{1}$ for $1\leq i\leq n$, we conclude that $\sigma_{n+1}$ is different from $\sigma_i$ for $1\leq i\leq n$. 

Since property (T) groups only admit finitely many irreducible representations of any given finite dimension by \cite[Theorem~2.6]{fd:reps:T}), we  thus necessarily have $k_n\coloneqq \dim(\sigma_n)\to\infty$. Define $\rho_{n+1}$ to be $\lambda_{m_n}\circ\vphi_{m_n}$ with its range restricted to this irreducible component. It is now straightforward to check that $\rho_n$ and $\sigma_n$ satisfy all requirements from the lemma. 
\end{proof}

Before proving the main theorem of this section (Theorem~\ref{thm:T}), we need the following construction from \cite{BL:T}. Recall that $\Sigma_n : \Sym(n) \to \U(n)$ is the standard permutation representation (Lemma \ref{lem:StoU}).

\begin{lemma}[{\cite[Proposition~2.3]{BL:T}}]\label{lem:BLTperturb}
    Let $\Gamma$ be a countable group with a fixed finite generating set $S$, $\pi:\bF_S\to\Gamma$, and let $\sigma:\Gamma\to \U(n)$ be a representation of $\Gamma$. Then there exists a homomorphism $\tau:\bF_S\to \U(n-1)$ such that
    \begin{enumerate}
        \item $\forall w\in \ker(\pi): \norm{\tau(w)-\mathbbm{1}}_2\leq \frac{3^{\abs{w}}}{\sqrt{n-1}}$.
        \item The inclusion map $\iota:\C^{n-1}\to \C^n$ satisfies for all $s\in S$:
        \[
        \norm{\sigma(s^{-1})\circ\iota\circ\tau(s) - \iota}_2 \leq \frac{2}{\sqrt{n}}.
        \]
    \end{enumerate}
Moreover, if $\sigma$ is of the form $\Sigma_n \circ \alpha$ for some $\alpha : \bF_S \to \Sym(n)$, then $\tau$ can be chosen to be of the form $\Sigma_{n-1} \circ \beta$, for some $\beta : \bF_S \to \Sym(n-1)$.
\end{lemma}

We will now use this to prove Theorem~\ref{thm:T}.

\begin{proof}[Proof of Theorem~\ref{thm:T}]
(1) Assume $\Gamma$ is hyperlinear and locally HS-stable. If $\Gamma$ is finitely presented, then $\Gamma$ is HS-stable by Lemma~\ref{lem:fp:local}, and the result follows from \cite[Theorem~1.3]{BL:T}. Assume $\Gamma$ is not finitely presented. Then by Shalom's Theorem \cite{fpT}, we can find a finitely presented property (T) group $\Gamma_0$ and a surjective homomorphism $f:\Gamma_0\to\Gamma$. Identify a fixed generating set $S=f(S)$ for $\Gamma_0$ and $\Gamma$, and consider the corresponding marking $\pi:\bF_{S}\to \Gamma_0$ for $\Gamma_0$.

Consider $\rho_n:\Gamma\to \U(k_n)$ and $\sigma_n:\Gamma_0\to \U(k_n)$ as in Lemma~\ref{lem:irreps}. Applying Lemma~\ref{lem:BLTperturb}, we get homomorphisms
\[
\tau_n: \bF_S\to \U(k_n-1)
\]
such that
\begin{itemize}
    \item $\forall w\in \ker(\pi): \norm{\tau_n(w)-\mathbbm{1}}_2\leq \frac{3^{\abs{w}}}{\sqrt{k_n-1}}$, and
    \item The inclusion maps $\iota_n:\C^{k_n-1}\to \C^{k_n}$ satisfy for all $s\in S$:
    \begin{equation}\label{eq:perturb}
    \norm{\rho_n(s^{-1})\circ\iota_n\circ\tau_n(s) - \iota_n}_2 = \norm{\sigma_n(s^{-1})\circ\iota_n\circ\tau_n(s) - \iota_n}_2 \leq \frac{2}{\sqrt{k_n}}.
    \end{equation}
\end{itemize}
In particular, it follows as in Lemma~\ref{lem:lift:descend} that $\tau_n$ yields an approximate homomorphism $\tau_n:\Gamma\to \U(k_n-1)$. By assumption $\Gamma$ is locally HS-stable, hence we can find a partial homomorphism
\[
\vphi_n:\Gamma\to \U(k_n-1)
\]
such that for all $g\in\Gamma$
\[
\norm{\tau_n(g) - \vphi_n(g)}_2\to 0.
\]
Since $\vphi_n$ is a partial homomorphism and $\Gamma_0$ is finitely presented, we can adjust $\vphi_n$ outside the ball $B(n)$ of radius $n$ so that for large enough $n$, $\vphi_n\circ f:\Gamma_0\to \U(k_n-1)$ is a genuine homomorphism. Together with \eqref{eq:perturb}, we then see that for every $s\in S$:
\[
\norm{\sigma_n(s^{-1})\circ\iota_n\circ \vphi_n(f(s)) - \iota_n}_2 \leq \eps_n
\]
for some $\eps_n\to 0$. Since $\Gamma_0$ has property (T), it thus follows that there exist morphisms of representations $\theta_n:\C^{k_n-1}\to\C^{k_n}$ from $\vphi_n\circ f$ to $\sigma_n$ such that 
\begin{equation}\label{eq:morphisms}
\norm{\iota_n - \theta_n}_2\to 0,
\end{equation}
see for instance \cite[Lemma~2.5]{BL:T}. However, applying Schur's Lemma implies that $\theta_n = 0$, contradicting \eqref{eq:morphisms}. This finishes the proof of (1). 

The proof of (2) proceeds similarly, upon using the second part of Lemma~\ref{lem:irreps} and the moreover part of Lemma~\ref{lem:BLTperturb}, as well as \cite[Proposition~2.4(ii)]{BL:T} instead of Schur's Lemma. We leave the details to the interested reader.
\end{proof}

\footnotesize

\bibliographystyle{alpha}
\bibliography{ref}

\end{document}